\pdfoutput=1
\documentclass[12pt,reqno]{amsart}
\usepackage[T1]{fontenc}
\usepackage[utf8]{inputenc}
\usepackage{amssymb,amsmath,amsthm,mathrsfs,array,graphicx,bbm,enumerate,wasysym,dsfont}
\usepackage[dvipsnames]{xcolor}
\usepackage[all]{xy}
\usepackage{bbm}
\usepackage{fullpage}
\usepackage{float}
\usepackage{verbatim}
\usepackage{hyperref}

\newtheorem{thm}{Theorem}
\newtheorem{lemma}[thm]{Lemma}
\newtheorem{rem}[thm]{Remark}
\newtheorem{defn}[thm]{Definition}
\newtheorem{prop}[thm]{Proposition}

\newtheorem{claim}[thm]{Claim}
\newtheorem{con}[thm]{Conjecture}

\numberwithin{equation}{section}

\newcommand{\Z}{\mathbb Z}
\newcommand{\R}{\mathbb R}
\newcommand{\C}{\mathbb C}

\newcommand{\N}{\mathbb N}
\newcommand{\D}{\mathbb D}
\newcommand{\E}{\mathbb E}

\newcommand{\closure}[1]{\overline{#1}}
\renewcommand{\P}{\mathbb P}
\renewcommand{\1}{\mathbf 1}
\newcommand{\A}{\mathds A}

\renewcommand{\L}{\mathcal L}
\newcommand{\F}{\mathcal F}

\renewcommand{\epsilon}{\varepsilon}

\newcommand{\eps}{\epsilon}	

\newcommand{\Int}{\operatorname{Int}}
\newcommand{\sgn}{\operatorname{sign}}

\newcommand{\titus}[1]{\textcolor{ForestGreen}{#1}}

\begin{document}
\title{Excursion decomposition of the 2D continuum GFF}
\date{ }
\author{Juhan Aru \and Titus Lupu \and Avelio Sep\'ulveda}

\begin{abstract}
{In this note we show that the 2D continuum Gaussian free field (GFF) admits an excursion decomposition that is on the one hand similar to the classical excursion decomposition of the Brownian motion, and on the other hand can be seen as an FK representation of the continuum GFF.} In particular, 2D continuum GFF can be written as an infinite sum of disjoint positive and negative sign excursions, which are given by Minkowski content measures of clusters of a critical 2D Brownian loop soup with i.i.d. signs. Although the 2D continuum GFF is not even a signed measure, we show that the decomposition to positive and negative parts is unique under natural conditions.

\end{abstract}


\maketitle

\section{Introduction}
The 2D continuum Gaussian free field (GFF) is a universal model of a continuum height function and has become a central object in the study of conformally invariant continuum random geometry.  The main reason for this is its strong connections with other objects like for example Schramm-Loewner Evolution, Brownian loop soup and Liouville quantum gravity measures (see e.g. overviews \cite {GwyHolSun, PowWer,BerPow}) and several known or conjectured convergence results towards the Gaussian free field \cite{NS, Kenyon_GFF, RiderVirag,BLR}.

In this note, we explain how to prove a decomposition of the 2D continuum Gaussian free field into an (infinite) sum of signed measures with disjoint supports. This decomposition is unique under natural conditions and can be obtained as a scaling limit of an honest excursion decomposition of the metric graph GFF. Thus our result says that there is a natural decomposition of the GFF into negative and positive parts, despite the fact that the field is not pointwise defined and not even a signed measure. On the one hand, the obtained decomposition shares many properties with the classical excursion decomposition of Brownian motion \cite{Ito} (but also exhibits some new surprising ones). {On the other hand, our decomposition can be also seen as an FK representation of the continuum GFF.}

We work in an open bounded simply-connected domain
$D\subset \C$, and we consider $\Phi$ a zero boundary Gaussian free field
on $D$.
To fix a normalization, we consider the GFF as the field coming from the following
functional integral
$$\exp\Big(-\dfrac{1}{2}\int_{D}\Vert\nabla\varphi\Vert^{2}\Big)\mathcal{D}\varphi.$$
More precisely, $\Phi$ is the centred Gaussian process with covariance  given by the Dirichlet Green's function $G_{D}(z,w)$ function
with the following divergence on the diagonal
$$G_{D}(z,w)\sim \dfrac{1}{2\pi}\log \vert z-w\vert^{-1}.$$
With this normalization, the value of the height gap (used later and introduced in \cite{SchSh}) is
$2\lambda = \sqrt{\pi/2}$.

The main contribution of this paper comes in three theorems: first we state the existence and uniqueness of an excursion decomposition, second we list properties of this decomposition, that mirror strongly those of the excursion decomposition of the one dimensional Brownian motion and make connections with the 2D critical Brownian loop soup. Finally, we show that the naturally defined excursion decomposition of the \textit{metric} GFF converges to the excursion decomposition of the continuum GFF. {Further contributions are Proposition \ref{prop:GFFspin}, which describes the continuum GFF as a rescaled limit of a random field spin model and explains the FK-representation point of view for the decomposition, Conjecture \ref{conj:dgff} that predicts what should be the continuum limit of the excursion decomposition of the \textit{discrete} GFF and Proposition \ref{prop:perco} that obtains uniform continuity of crossing probabilities of annuli by sign clusters of the metric graph GFF.} We make use of known couplings between GFF, CLE$_4$ and Brownian loop soup \cite{SchSh, SheffieldWerner2012CLE, ASW, QW2018, ALS1, ALS2} and build on techniques introduced in \cite{SchSh2, ASW, ALS1, ALS2}; the most technical part of the paper is the proof of uniqueness. 

The existence of the excursion decomposition is given in the following theorem. 
\begin{thm}[Excursion decomposition of the 2D Gaussian free field]\label{thm:excdcmp}
Let $\Phi$ be a zero boundary GFF in $D$. There exists a unique collection of positive measures $(\nu_k)_{k \geq 1}$ with supports $(C_k)_{k \geq 1}$, and a collection of signs $(\sigma_k)_{k \geq 1}$, such that the following conditions hold:
\begin{enumerate}
    \item We can write 
    \begin{equation}\label{Eq decomp 1}
    \Phi = \lim_{N\to \infty}\sum_{k = 1}^N \sigma_k \nu_k,
    \end{equation}
    where the sum is ordered by decreasing size of the diameter of $C_k$. The sum converges almost surely in all the Sobolev spaces $H^{-1-\varepsilon}(D)$ (i.e. for the Sobolev norms) for $\varepsilon>0$.
    \item The decomposition satisfies the following Markov property. For any smooth simple path $\gamma \subset \closure{D}$, starting from the boundary, let $\gamma^{exc}$ denote the closure of the union of all sets $C_k$ that intersect $\gamma$. We can write  almost surely $
    \Phi = \Phi^{\gamma^{exc}} + \Phi_{\gamma^{exc}},$
    with
    \begin{align*}
        \Phi_{\gamma^{exc}} = \sum_{k: C_k\cap \gamma \neq \emptyset} \sigma_k \nu_k,
    \end{align*}
    where the sum is again ordered by decreasing size of diameter of $C_k$ and converges almost surely in all the Sobolev spaces $H^{-1-\varepsilon}(D)$,
    for $\varepsilon>0$. Further, conditionally on $\gamma^{exc}$, the field $\Phi^{\gamma^{exc}}$ is independent of $\Phi_{\gamma^{exc}}$ and has the law of a zero boundary GFF in the domain $D\backslash \gamma^{exc}$.
    \item The collection $(C_k)_{k \geq 1}$ is pairwise disjoint, locally finite\footnote{Locally finite means that for any $\epsilon>0$ there are finitely many $C_k$ with diameter bigger than $\epsilon$.}, and further each $C_k$ is connected.
\end{enumerate}
We call $(C_k)_{k \geq 1}$ the (sign) excursion clusters, $(\nu_k)_{k \geq 1}$ the sign excursions and the triplet $(C_k, \sigma_k, \nu_k)_{k \geq 1}$ the excursion decomposition of $\Phi$.
\end{thm}
Further properties of the excursion decomposition are listed in the following theorem.
\begin{thm}[Properties of the excursion decomposition]\label{thm:prop}
Let $\Phi$ be a zero boundary GFF in $D$ and $(( C_k,\nu_k, \sigma_k))_{k \geq 1}$ respectively the excursion clusters, the measure and their signs in the excursion decomposition of Theorem \ref{thm:excdcmp}. Then, the following properties hold:
\begin{enumerate}
    \item The excursion decomposition $(( C_k,\nu_k, \sigma_k))_{k \geq 1}$ is measurable w.r.t. $\Phi$.
    \item In the joint law of $(( C_k,\nu_k, \sigma_k))_{k \geq 1}$, the signs $(\sigma_k)_{k \geq 1}$ are independent of the rest and have the law of i.i.d. Rademacher random variables.
    \item For all $k \geq 1$, the measures $\nu_k$ are given by Minkowski content measure of $C_k$ defined by
$$\nu_k(f) := \lim_{r \to 0}\frac{1}{2}\vert \log r\vert^{1/2}\int_Df(z)1_{d(z, C_k)\leq r}dz,$$ for all $f \in C(D)$. In particular for all $k \geq 1$, $\nu_k$ is determined by $C_k$.
\item The law of $(C_k)_{k \geq 1}$ equals to that of 
(topological closures of) clusters of a 2D Brownian loop soup at the critical intensity $\alpha = 1/2$ in $D$, ordered by decreasing diameter.
{Further, the collection of outer boundaries of the outermost clusters $(C_k)_{k \geq 1}$ has the law of $CLE_4$.}
\end{enumerate}
\end{thm}

Further, one can justify the name excursion decomposition by showing a convergence result from the well defined excursion decomposition on the metric graph GFF. See Section \ref{s:convergence} for the exact set-up and {see Conjecture \ref{conj:dgff} for the case of the discrete GFF, whose excursion decomposition, we believe, converges to a different continuum decomposition where individual `excursions' are still positive and negative measures, but signs are no longer independent, they rather alternate with nesting.}

\begin{thm}[Convergence of the excursion decomposition]\label{thm:conv a}
Let $\Phi$ be a zero boundary GFF on $D$ and $\tilde \phi_n$ be a sequence of zero boundary metric graph GFFs on $\widetilde D_n$ that are coupled with a GFF $\Phi$ such that a.s. $\tilde \phi_n \to \Phi$ in $H^{-\epsilon}(D)$, for some $\epsilon>0$.
Further, take $(\widetilde C_k^n,  \tilde\nu_k^n,\tilde\sigma_k^n)_{k \geq 1}$ the excursion decomposition of $\tilde \phi_n$. 

We have that for every $k>0$,  $ \widetilde C_k^{(n)}\to C_k$, 
$\tilde\nu_k^{(n)}\to \nu_k$ and $\tilde\sigma_k^{(n)}\to \sigma$ as $n\to \infty$, where the convergence is in probability and in the Hausdorff topology for the first component, and in the weak topology of measures for the second component.
\end{thm}

Let us elaborate on these theorems via some further remarks.
\begin{enumerate}[i.]
    \item  It is known that the 2D continuum Gaussian free field is not a signed measure and in particular it cannot be written as a difference of two sigma-finite positive measures. Thus such a rewriting as a sum of disjoint signed measures is in itself already non-trivial. 
\item Previously a similar decomposition was known for the continuum limit of the magnetization field of the critical 2D Ising model,
in which case it is the continuum analogue of the standard FK representation of the lattice Ising model.
The continuum limit of the magnetization field was constructed in
\cite{CGN}.
The continuum FK decomposition was conjectured in \cite{CamiaNewman09PNAS}, and the proof finalized in \cite{CamiaJiangNewman}.
\footnote{We are grateful to F. Camia for clarifying this history of the Ising decomposition.} 
A crucial input was the detailed understanding of the scaling limits of interfaces and correlation functions of the critical Ising and FK-Ising models
\cite{CDCHKS,CDCH,CHIcorr,ChelkakSmirnov,Smirnov}.
As a difference to the free field case, 
in the Ising setting the renormalised area measures are constructed via a convergence argument from the discrete area measures and in the continuum limit the excursion decomposition is not measurable with respect to the continuum magnetization field.
This failure of measurability comes from the fact that two different continuum FK clusters can touch and there are several ways to split an Ising spin cluster into FK clusters.
    \item {In fact, similar to the Ising case mentioned above, also our excursion decomposition can be seen as a continuum FK representation of the GFF. Indeed, as observed in \cite{LupuWerner16Ising}, the sign clusters of the metric graph GFF $\tilde \phi_n$ can be seen as a certain FK representation for the random-field Ising model of the discrete GFF given by $\operatorname{sign}(\phi_n)$. Proposition \ref{prop:GFFspin} shows that the continuum GFF is a renormalized scaling limit of this model, making the FK-viewpoint more precise. Interestingly, this FK decomposition which is not measurable w.r.t. discrete GFF, becomes measurable in the continuum limit. }
    \item {Related to the previous comment, we believe that the excursion decomposition of the discrete GFF does not converge to our continuum decomposition - see Conjecture \ref{conj:dgff} for a precise statement. This alternative continuum decomposition of the continuum GFF does not satisfy equally nice independence properties, e.g. the signs of the sign excursions are not independent. }
    \item The existence of a decomposition of the 2D continuum GFF into a signed sum of measures (without uniqueness, measurability and an explicit description of the structure of the decomposition) could be also obtained using subsequential convergence results from the metric graph, using results from 
    \cite{LupuConvCLE} but no further SLE theory.
    \item To prove existence and uniqueness of the decomposition we only need to use basic properties of the GFF and its local sets (including CLE$_4$, SLE$_4$), and thus in particular we do not use isomorphism theorems. In fact also the excursion clusters have a writing in terms of only the nested CLE$_4$: see Remark \ref{rem:cle}.
    \item We expect the existence and uniqueness of the decomposition, and all the properties to hold also in non-simply-connected domains. However, it adds some technicalities that we decided not to address in this work.
    \item The convergence of the sum can most likely be improved to $H^{-\eps}$ for all $\eps > 0$.
     \item For the convergence in \eqref{Eq decomp 1}, the compensation induced by the sign is crucial,
and the total variation measure $\Sigma_{i\geq 1}\nu_{i}$ diverges in every open subset of $D$. There is some freedom in the specific order on the clusters
$(C_i)_{i \geq 1}$. However, it is important to fix the order independently of the signs $(\sigma_i)_{i \geq 1}$.  Notice that we do not a priori ask any independence properties of the signs, and obtain them as a corollary.
    \item {One may wonder what would be the minimal assumption to have uniqueness of the decomposition above. However, the answer might not be straightforward: 1) as already mentioned, we expect there to be another natural excursion decomposition that comes from the convergence of the excursion decomposition of the discrete GFF 2) as soon as one works in classes of irregular functions such decompositions are in general not unique without further assumptions. Indeed, even for example the Hahn decomposition of signed measures is unique only up to measure $0$ sets. Or, for a concrete example, consider the case of Brownian motion on $[0,1]$, but seen as a probability measure of $L^2([0,1])$. Now let's look for decompositions of $[0,1]$ into closed connected disjoint sets where Brownian Motion is either non-positive or non-negative. It is easy to see that, unless we invoke some extra conditions - like a certain Markov property after discovering some excursions, or independence of signs of excursions, or possibly some maximality property -, we can in addition to the natural decomposition, where we take the support of each excursion, also find many others. Indeed, we can always first take the natural decomposition, but then further write any of these closed intervals as a countable union of smaller closed intervals, up to a zero measure set. As long as we work in $L^2([0,1])$, the remaining zero measure set can be just forgotten and all the above-listed conditions would be satisfied.}
\item {This theorem can be further tweaked to write the 2D continuum GFF using Poisson point processes of excursions very similarly to the classical writing of the Brownian motion by concatenating a Poisson point process (PPP) of Brownian excursions. Indeed, as shown in \cite{WernerWu2013ExplorCLE}, one can define an infinite measure on the space of loops pinned at a uniform point on the boundary, such that the whole CLE$_4$ can be constructed using a single PPP with this intensity measure - see \cite{WernerWu2013ExplorCLE} Section 4, or \cite{AS2} Section 6.1 for a more detailed explanation. Now, as mentioned, CLE$_4$ gives only the outer boundaries of the outermost excursion clusters, but one can further include the clusters and the sign measures in the above-mentioned intensity measure to obtain an intensity measure for clusters pinned at boundary. This way one obtains a way to sample all the outermost clusters via a PPP; one further iterates in the interiors of each cluster to get the full decomposition.}
\item It would be very interesting to see similar decompositions for other random distributions and indeed, Jego, Lupu and Qian manage to prove similar decompositions for random fields constructed from sub-critical Brownian loop soups \cite{JLQ23b}. Among other things, they also give an alternative proof for existence of the decomposition in the critical case that does not rely local sets of the GFF (nor CLE$_4$, SLE$_4$), but that does not provide uniqueness and measurability.
\end{enumerate}

The rest of this note is structured as follows: we collect definitions of main objects in Section 2; in Section 3 we prove the existence part of Theorem 1 and deduce the properties of Theorem 2. In Section 4 we prove the uniqueness of the decomposition and in Section 5 the convergence. Finally, in Section 6 
we discuss several further aspects: firstly, we prove  uniform continuity of crossing probabilities on metric graphs.
We then explain how to see the continuum GFF as a scaling limit of spin models and why our decomposition can be alterantively seen as a FK representation.
In the same section we also discuss the conjectured scaling limit of the excursion decomposition of the discrete GFF.

\subsection*{Acknowledgement} The authors are thankful two Wendelin Werner for many useful and inspiring discussions at their time in ETH. The research of J.A. is supported by Eccellenza grant 194648 of the Swiss National Science Foundation and he is a member of NCCR Swissmap. The research of A.S is supported by Grant ANID AFB170001, FONDECYT iniciación de investigación N° 11200085 and ERC 101043450 Vortex.

\section{Definitions and preliminaries}

For the convenience of the reader, we collect here the definitions of the 2D continuum Gaussian free field and its local sets, CLE$_4$ and Brownian loop soup. For more information, see e.g. preliminaries of \cite{ALS1, ALS2} or the book \cite{PowWer}. 

The continuum Gaussian free field (GFF) is the generalisation of Brownian motion, replacing the time axis by a $d$-dimensional domain. More precisely, it is defined as follows.
\begin{defn}[Gaussian free field]
Let $D\subseteq \C$ denote a finitely connected domain. The $2$-dimensional zero boundary continuum GFF in $D$ is the centred Gaussian process $(\Phi,f)_{f\in C_c^\infty(\C)}$ whose covariance is given by 
$$\E\left[(\Phi,f)(\Phi,g)\right] = \int\int_{D\times D} f(z)G^{D}(z,w)g(w)\, dz dw;\quad f,g\in C_c^\infty(\C),$$
where $G^{D}$ denotes the zero boundary Green's function for the Laplacian in $D$. 
\end{defn}

For any open set $U$ that is a union of countably many finitely-connected domains, we define the zero boundary GFF on $U$ as a disjoint union of independent zero boundary GFFs in the connected components. The GFF is almost surely in $H^{-\eps}(U)$ for any $\eps > 0$, but we can also consider the GFF as a random distribution on larger domains $U' \supseteq U$, extending it outside of $U$ by zero. The continuum GFF can be essentially characterized by its Markov property \cite{BPR,BPR2,GFFchar} and random sets coupled with the GFF that satisfy a strong Markov property are called local sets. For a more general discussion of local sets and their properties we refer to 
	\cite{Aru, SchSh2,PowWer}.
	
\begin{defn}[Local sets]\label{d: local sets}
Consider a random triple $(\Phi, A,\Phi_A)$, where $\Phi$ is a  GFF in $D$, $A$ is a random closed subset of $\overline D$ and $\Phi_A$ a random distribution that can be viewed as a harmonic function when restricted to
$D \backslash A$.
We say that $A$ is a local set for $\Phi$ if conditionally on $(A,\Phi_A)$, $\Phi^A:=\Phi - \Phi_A$ is a  GFF in $D \backslash A$. 
\end {defn} 

We list here some properties of local sets that we use implicitly or explicitly, see for instance \cite {SchSh2,Aru} for derivations and further properties. 

\begin{lemma}\label{lem:BPLS}  The following properties hold for local sets of the GFF. 
			\begin {enumerate}
			\item Any local set can be coupled in a unique way with a given GFF: Let $(\Phi,A,\Phi_A,\widehat \Phi_A)$ be a coupling, where $(\Phi,A,\Phi_A)$ and $(\Phi,A,\Phi'_A)$ satisfy the conditions of this definition. Then, a.s. $\Phi_A= \Phi'_A$. Thus, being a local set is a property of the coupling $(\Phi,A)$, as  $\Phi_A$ is a measurable function of $(\Phi,A)$. 
			\item If $A$ and $B$ are local sets coupled with the same GFF $\Phi$, and $(A, \Phi_A)$ and $(B, \Phi_B)$ are conditionally independent given $\Phi$, then $A \cup B$ is also a local set coupled with $\Phi$ and the boundary values of $\Phi_{A \cup B}$ agree with those of $\Phi_B$ or $\Phi_A$ at every point of the boundary of $A \cup B$ that is of positive distance of $A$ or $B$ respectively\footnote{We say that $\Phi_{A \cup B}$ agrees with $\Phi_A$ at a point $x\in \partial(A\cup B) \cap \partial A $ if for any sequence of $x_n \notin A\cup B$ converging to $x$, $\Phi_{A \cup B}(x_n)-\Phi_{A} (x_n) \to 0$ as $n\to \infty$.}. Additionally, $B\backslash A$ is a local set of $\Phi^A$ with $(\Phi^A)_{B\backslash A} = \Phi_{B\cup A}-\Phi_{A}$  . 
			\item Let $(\Phi, (A_n)_{n\in \N},(\Phi_{A_n}))_{n\in \N}$ a sequence of conditionally independent local sets coupled with the same GFF $\Phi$. Furthermore, assume that $A_n$ is increasing. Then $A_\infty = \overline{\bigcup_{n\in \N} A_n}$ is a local set. Furthermore, if a.s. for all $n\in \N$, $A_n$ is connected to the boundary, then a.s. $\Phi_{A_n}\to \Phi_{A}$.
		\end{enumerate}
	\end{lemma}

In particular, we will use the existence and uniqueness of the following type of local sets: two-valued local sets introduced in \cite{ASW} and studied in \cite{ALS1}, and first passage sets, introduced in \cite{ALS1,ALS2}. For definitions of thin local sets, bounded type local sets we refer e.g. to \cite{ASW, ALS3}.

\begin{thm}[Two-valued local sets: existence and uniqueness]\label{thm:tvs}
Let $a, b > 0$ be such that $a+b \geq 2\lambda$. Then one can couple a thin\footnote{Thin means that $\Phi_A$ is a.s. equal to a harmonic function everywhere, see \cite{Se}.} bounded type local set $\A_{-a,b}\neq \emptyset$ with a GFF $\Phi$ such that in each connected component $O$ of $D\backslash \A_{-a,b}$ the harmonic function $\Phi_A$ is equal to either $-a$ or $b$. Moreover, the sets $\A_{-a,b}$ are
	\begin{itemize} 
		\item unique in the sense that if $A'$ is another BTLS coupled with the same $\Phi$,  such that a.s. it satisfies the conditions above,	then $A' = \A_{-a, b}$ almost surely;  
		\item measurable functions of the GFF $\Phi$ that they are coupled with;
		\item monotone in the following sense: if $[-a,b] \subset [-a', b']$ with $b+a \ge 2\lambda$,  then almost surely, $\A_{-a,b} \subset \A_{-a', b'}$. 
	\end{itemize}
\end{thm}
It was observed in \cite{ASW} that the Minkowski dimension of all of any two-valued set is a.s. strictly smaller than $2$ (the a.s. Hausdorff dimension was precisely calculated in \cite{SSV}); we will make use of this fact for this for $\A_{-2\lambda, 2\lambda}$. 

Two-valued local sets are of importance for us as the boundaries of sign excursions in our decomposition are given by iterating two-valued local sets $\A_{-2\lambda, 2\lambda}$. The excursion clusters themselves are given by first passage sets.

\begin{defn}[First passage set]\label{Def ES}
		Let $a\in \R$ and $\Phi$ be a GFF in $D$. We define the first passage set of $\Phi$ of level $-a$ as the local set of $\Phi$ such that $\partial D \subseteq \A_{-a}$, with the following properties:
		\begin{enumerate}
			\item Inside each connected component $O$ of $D\backslash \A_{-a}$, the harmonic function $\Phi_{\A_{-a}}\mid_{D\backslash \A_{-a}}$ is equal to $-a$. 		
			\item $\Phi_{\A_{-a}}+a\geq 0$, i.e., for any smooth positive test function $f$ we have 
			$(\Phi_{\A_{-a}}+a,f) \geq 0$, in other words $\nu:=\Phi_{\A_{-a}}+a$ is a positive measure with support $\A_{-a}$. 
		\end{enumerate}	
	\end{defn}
		
 The key result is the following.
	\begin{thm}\label{thm:FPS}[Theorem 4.3 and Proposition 4.5 of \cite{ALS1}, Proposition 5.7 of \cite{ALS2}]\label{Thm::FPS}For all $a\geq 0$,  the first passage set, $\A_{-a}$, of $\Phi$ of level -a exists and satisfies the following properties:
		\begin{enumerate}
			\item Uniqueness: if $A'$ is another local set coupled with $\Phi$ and satisfying Definition \ref{Def ES}, then a.s. $A'=\A_{-a}$.
			\item Measurability: $\A_{-a}$ is a measurable function of $\Phi$.
			\item Monotonicity: If $a\leq a'$, then $\A_{-a}\subseteq \A_{-a'}$
			\item \titus{Local} finiteness: for any $\eps > 0$ there are only finitely many connected components of $D \setminus \A_{-a}$ of diameter larger than $\eps$.
		\end{enumerate}
	\end{thm}

\subsection{Couplings between different objects}

It was shown in \cite{SchSh, SchSh2} that SLE$_4$ can be seen as a contour line of the continuum GFF. Miller \& Sheffield \cite{MS} discovered that also CLE$_4$ can be coupled as with the GFF. In \cite{ASW} this latter coupling was rephrased in the language of two-valued sets - the two-valued set $\A_{-2\lambda,2\lambda}$ has the law of a CLE$_4$ carpet. 

\begin{thm}[Section 4 of \cite{ASW}]\label{thm:CLE4_coupling}
Let $\Phi$ be a GFF in $D$ and $\A_{-2\lambda,2\lambda}$ be its TVS of levels $-2\lambda$ and $2\lambda$. Then $\A_{-2\lambda, 2\lambda}$ has the law of CLE$_4$ carpet. Moreover, it satisfies the following properties:
		\begin{enumerate}
			\item The loops of $\A_{-2\lambda,2\lambda}$ (i.e. the boundaries of the connected components of $D \backslash \A_{-2\lambda, 2\lambda}$) are continuous simple loops. $\A_{-2\lambda, 2\lambda}$ is the closure of the union of all loops.
			\item The collection of loops of $\A_{-2\lambda,2\lambda}$ is locally finite, i.e. for any $\eps>0$  there are only finitely many loops that have diameter bigger than $\eps$.
			\item Almost surely no two loops of $\A_{-2\lambda,2\lambda}$ intersect, nor does any loop intersect the boundary; also almost surely every fixed point is surrounded by some loop.
			\item The conditional law of the labels of the loops of $\A_{-2\lambda,2\lambda}$ given $\A_{-2\lambda,2\lambda}$ is that of i.i.d. random variables taking values $\pm 2\lambda$ with equal probability.
		\end{enumerate}
\end{thm}

From the ground-setting work of Sheffield and Werner, we know further that in simply-connected domains CLE$_4$ loops can be described using the critical Brownian loop soup (BLS).
\begin{thm}[Theorem 1.6 in \cite{SheffieldWerner2012CLE}]\label{clebls}
Let $D$ be a simply-connected domain and consider the critical Brownian loop-soup $\mathcal{L}$ in $D$. Then CLE$_4$ loops are exactly the outer boundaries of the outermost clusters of this Brownian loop soup. 	\end{thm}
This theorem together with Theorem \ref{thm:CLE4_coupling} implies the following Markov property for $\A_{-2\lambda,2\lambda}$
\begin{prop}\label{prop:Markov CLE}Let $\Phi$ be a GFF in a simply connected domain $D$ and $\gamma:[0,1]\mapsto \overline D$ be a simple continuous curve such that $\gamma(0)\in \partial D$ and $\gamma((0,1))\subseteq D$. Define $\gamma^{ext}$ the closure of the union of all loops of a $\A_{-2\lambda,2\lambda}$ that intersect $\gamma$. We have that $\gamma^{ext}$ is a BTLS of $\Phi$, where $\Phi_{\gamma^{ext}}$ can be characterised as follows. Take $\mathcal I$ the union of the interior all loops $\ell$ of $\A_{-2\lambda,2\lambda}$ that intersect $\gamma$, then $\Phi_{\gamma^{ext}}(z)=\pm 2\lambda$ for any $z\in \mathcal I$ and $\Phi_{\gamma^{ext}}(z)=0$ for all $z\in D\backslash \overline{\mathcal I}$.
\end{prop}

In fact, the relation of Theorem \ref{clebls} can be further strengthened. First in \cite{QW15} the authors show that one can couple the critical Brownian loop soup, CLE$_4$ and the zero boundary GFF on the same probability space such that CLE$_4$ describes the outer boundaries of outermost BLS clusters as above and the Wick square of the GFF equals the renormalised occupation time of the BLS. We will not use this statement directly, however we use a certain strengthening that further identifies the Brownian loop soup clusters given their boundary with first passage sets defined and constructed in \cite{ALS1}.
\begin{prop}[Corollary 5.4 in \cite{ALS2}]\label{thm:BLSGFF2}
			\label{CorClusterLoopSoup}
			Let $D$ be a simply connected domain. 
			Conditionally on the outer boundary $\ell$ of a Brownian loop-soup cluster in $\L^{D}_{1/2}$, the topological closure of the cluster itself is distributed like a first passage set $\A_{-2\lambda}$ inside 
$\operatorname{Int}(L)$, the interior surrounded by 
$\Upsilon$.
		\end{prop}
	
Finally, it was observed in \cite{ALS1} that one can identify the GFF restricted to a first passage set by its Minkowski content measure.

\begin{thm}[Theorem 5.1 in \cite{ALS1}]\label{Min}
	Let $D$ be simply-connected and $\Phi$ a GFF and suppose $\A_{-a}$ is a first passage set of level $-a$. Writing $\Phi = \Phi_{\A_{-a}} + \Phi^{\A_{-a}}$ as in Definition \ref{d: local sets}, we obtain the following. 
	The measure $\nu_{\A_{-a}}:=\Phi_{\A_{-a}}+a$ is a measurable function of 
	$\A_{-a}$. Moreover, it is proportional to the Minkowski content measure in the gauge  
	$r\mapsto \vert \log(r)\vert^{1/2} r^{2}$. More precisely, almost surely for any continuous $f$ compactly supported in $D$,
	\[\nu_{\A_{-a}} = \lim_{r \to 0}\frac{1}{2}
	\vert\log(r)\vert^{1/2} \int_D f(z)  \1_{d(z,\A_{-a})\leq r}dz.\]	
\end{thm}

\section{Existence of the excursion decomposition and its properties} \label{s:existence}

In this section, we prove the existence of the excursion decomposition together with the properties stated in Theorem \ref{thm:prop}. These both follow rather directly from the theory of bounded type local sets and first passage sets of the GFF, though some care is needed in collecting and combining the results and techniques and in taking care of lack of absulte convergence.

We start by an elementary estimate on the $H^{-1}$ norm of a GFF on open strongly non-connected sets that can be written disjoint unions of open domains of small diameter. This lemma is used to show that contributions to the excursion decomposition coming from small excursions can be summed.
\begin{lemma}\label{lem:tail}
Suppose $\widehat D_n\subseteq D$ is a sequence of decreasing open set (not necessarily connected) such that the maximal diameter over its connected components goes to $0$ as $\epsilon\to 0$. Consider $\Phi^{\widehat D_n}$ a GFF in $\widehat D_n$. Then $\E \left[\|\Phi^{\widehat D_n}\|_{H^{-1}(D)}^2\right] \to 0$, as $n\to \infty$
\end{lemma}
\begin{proof}
    This follows from the dominated convergence theorem ($G_{\widehat D_n}\leq G_{D}$) and the computation
    \begin{align*}
    \E \left[\|\Phi^{\widehat D}\|_{H^{-1}(D)}^2\right]=\iint_{\widehat D\times \widehat D} G_D(x,y)G_{\widehat D} (x,y) dx dy\to 0, \ \ \ \text{ as } n\to \infty.\end{align*}
    
\end{proof}

We are now ready to prove the existence part of the main theorem.

\begin{proof}[Proof of the existence of an excursion decomposition in Theorem \ref{thm:excdcmp}]
We start by considering the coupling $(\Phi, \text{CLE}_4 = (\ell_k)_{k \geq k}, (\sigma_k)_{k \geq 1})$ between the GFF, CLE$_4$ loops and the i.i.d. signs coming from Theorem \ref{thm:CLE4_coupling}. We can order the loops in descending order of their diameter. Note that this theorem implies the almost sure equality
$$\Phi = \sum_{k \geq 1}2\lambda \sigma_k1_{ \Int(\ell_k)} + \Phi^{\Int(\ell_k)},$$
where given the CLE$_4$ loops $(\ell_k)_{k \geq 1}$, $\Phi^{\Int(\ell_k)}$ are independent zero boundary GFFs\footnote{Note that $\Phi^{\A_{-2\lambda,2\lambda}}= \sum_k \Phi^{\Int(\ell_k)}$} inside $\Int(\ell_k)$ and $(\sigma_k)_{k \geq 1}$ are i.i.d. Rademacher random variables. 
Using Lemma \ref{lem:tail} to control the tails, we can restrict our attention to the subset $J_\eps$ of $k\in \N$ such that the diameter of $C_k$ is at least $\epsilon > 0$. As the set of CLE$_4$ loops is locally finite, $J_\eps$ is finite.

Now, consider $k \in J_\eps$. Conditionally on $\ell_k$ the law of $\Phi$ restricted to $\Int(\ell_k)$ is equal to that of $2\lambda \sigma_k + \Phi^{\Int\ell_k}$, where the conditional law of $\Phi^{\Int(\ell_k)}$ is that of a GFF in $\Int(\ell_k)$. We now sample $A_k:=\A_{-2\lambda}(\sigma_k \Phi^{\Int(\ell_k)})$  and define the positive measure $\nu_k:= \sigma_k\Phi_{A_k}^{\Int(\ell_k)}+  2\lambda$. This measure is supported in $A_k$ thanks to Definition \ref{Def ES}.
We then have that
\begin{align*}
2\lambda \sigma_k + \Phi^{\Int(\ell_k)} = \sigma_k\nu_{\A_{-2\lambda}} + \Phi^{\Int(\ell_k)\setminus\A_{-2k}} .\end{align*}
Thus we obtain the following decomposition
$$\Phi = \sum_{k \in J_\eps}\left(\sigma_k\nu_{A_{k}} + \Phi^{\Int(\ell_k)\backslash A_{k}}\right) + \sum_{k' \notin J_\eps}\left(2\lambda \sigma_{k'}1_{\Int{\ell_{k'}}} + \Phi^{\Int(\ell_{k'})}\right),$$
where the outermost boundaries $\ell_k$ and $\ell_{k'}$ are ordered in the descending order of their diameter. We now iterate this process inside each connected component of $D \setminus \overline{\bigcup_{k \in J_\eps} A_k}$.

To write down the result of this iterative construction, we need to fix some notation. We denote the outermost loops and clusters of the $n-$th iteration that themselves have diameter larger than $\eps$ by $(\ell_{n, k})_{k \in J_{n, \eps}}, (C_{n,k})_{k \in J_{n, \eps}}$, having ordered them decreasingly by diameter, and the corresponding signs and Minkowski measures by $(\sigma_{n,k})_{k \in J_{n, \eps}}, (\nu_{n,k})_{k \in J_{n, \eps}}$. 

The iteration then gives us the following almost sure equality:
\begin{align*}
\Phi = \sum_{k \in \cup_{n \leq N} J_{n,\eps}}\left (\sigma_k\nu_{A_{n,k}} + \Phi^{\Int(\ell_{n,k}\backslash A_{n,k})} \right)+ \sum_{\substack{(n',k'):\\  k'\notin  J_{n',\eps}\\  n'\leq N}}\left( 2\lambda \sigma_k1_{\Int(\ell_k)}  + \Phi^{\Int(\ell_{n,k})}\right),\end{align*}
where again the ordering in the first finite sum is along decreasing size of the diameter. This writing allows us to apply Lemma \ref{lem:tail} directly to obtain an error of order $o_\epsilon(1)$ independently of the level of iteration $N$ on the second term. Part (1) of Theorem \ref{thm:excdcmp} now follows from the a.s. martingale convergence theorem and the fact that for any $\epsilon > 0$, there is almost surely a finite $N$ such that all loops of diameter larger than $\eps$ have been discovered {(for a single level this is just local finiteness of CLE$_4$, for the nested version see e.g. Theorem 1.5 in \cite{AruPaponPowell}).}

The properties listed in (3) of Theorem \ref{thm:excdcmp} for the excursion clusters $(C_k)_{k \geq 1}$ hold by construction. The Markov property follows from the following claim combined with the argument above that again shows we can sum the sign excursions in their decreasing order of diameter.
\begin{claim}

Consider $\gamma$ a smooth simple path in $\closure{D}$ starting from the boundary. 
    Let $(C_k, \ell_k)_{k \in I}$ be the collection of outermost clusters with $C_k \cap \gamma \neq \emptyset$ and denote their outer boundaries by $\ell_k$. Let $I_\eps$ denote the set of $k \in I$ for which the diameter of $C_k$ is at least $\epsilon$ and define $A_\eps = \closure{\cup_{k \in I_\eps} C_k} \cup \closure{ \cup_{k \in I} \ell_k}$. Then $A_\eps$ is a local set, such that $\phi_{A_\eps} = \sum_{i \in I_\eps}
    {\sigma_i}
    \nu_i + \sum_{i \in I \setminus I_\eps} 2\lambda {\sigma_i}1_{z \in \Int{L_i}} $. 
\end{claim}

\begin{proof} The claim follows directly from iterating the Markov property of $\A_{-2\lambda,2\lambda}$ in Proposition \ref{prop:Markov CLE}, together with the strong Markov property of FPS and the construction above.

\end{proof}

\end{proof}

We proceed to discuss further properties of the excursion decomposition, i.e. Theorem \ref{thm:prop}, assuming uniqueness of the decomposition. In essence, this amounts to handpicking a few interesting results from the literature.

We are now ready to give a proof of Theorem \ref{thm:prop}, assuming already uniqueness.

\begin{proof}[Proof of Theorem \ref{thm:prop}]

Properties (1) and (2) follow directly from the construction given above. Property (3) follows from the construction of the excursion clusters and excursions via First passage sets of height $\pm 2\lambda$ and Theorem \ref{Min}. 

The law of outer boundaries of outermost clusters is also clear from the construction. The identification with clusters of 2D Brownian loop soup follows further from iterating Theorem \ref{clebls} to identify the outer boundaries of critical BLS clusters with those of excursion clusters in the construction above, and Theorem \ref{thm:BLSGFF2} to identify the critical BLS clusters with the excursion clusters above.
\end{proof}

To finish this section, we explain here how the whole sign cluster could be seen by iterating CLE$_4$. Such iterations were first considered in \cite{Ai} to give a geometric martingale approximation of the Liouville measure; the relation to Brownian loop soup clusters became clear with \cite{ALS1, ALS2}.

\begin{rem}\label{rem:cle}[Sign cluster via nested CLE$_4$]{
We saw that conditionally on the outer boundary of an excursion cluster, the cluster itself is distributed as the FPS of height $2\lambda$. However, there is a way to obtain first passage sets using iterations of two-valued sets; for example, see Lemma 2.5 in \cite{APS} and the discussion under it. Indeed, the it is explained there that FPS of level $2\lambda$ can be obtained by iterating two-valued local sets $\A_{-2\lambda, 2\lambda}$ until every connected component of the complement has boundary conditions $2\lambda$. More precisely, we start by sampling $\A_{-2\lambda, 2\lambda}$ and we repeat the construction inside each loop which does not have the label $2\lambda$. This way we observe around each point a random walk with values in $2\lambda \Z$, stopped at reaching $2\lambda$. As $\A_{-2\lambda, 2\lambda}$ has the law of CLE$_4$, and the signs of labels $\pm 2\lambda$ are i.i.d., we have a way of describing the whole sign cluster using iterated CLE$_4$ via a structure of branching simple random walks.}
\end{rem}

\section{Uniqueness of the excursion decomposition}\label{S:uniqueness}

In this section, we prove the uniqueness part of Theorem \ref{thm:excdcmp}. Throughout this section $( \nu_k, C_k, \sigma_k)_{k \geq 1}$ denotes the excursion decomposition constructed in Section \ref{s:existence}, and $(\hat \mu_k, \hat C_k, \hat \sigma_k)_{k \geq 1}$ is another decomposition that satisfies the properties of Theorem \ref{thm:excdcmp} for the same GFF $\Phi$. By conformal invariance we may assume that we work in the unit disk $\D$ throughout this section.

The proof of uniqueness is dissected into following propositions. We first show that the excursion clusters in the construction of the previous section are in a certain sense minimal: 

\begin{prop}\label{prop:ss}
     Almost surely, for every $k\in \N$ there exists $\hat k(k) \in \N$ such that $C_k\subseteq \hat C_{\hat k(k)}$. 
\end{prop}

This already implies that almost surely each excursion decomposition has a cluster surrounding any fixed point of the domain.

Then, we show that the signs of intersecting sign clusters of the two decompositions introduced above have to match.
\begin{prop}\label{prop:es}
    Let $\hat k(k)$ be as in Proposition \ref{prop:ss}. Then almost surely, $\sigma_{k}=\hat \sigma_{\hat k(k)}$.
\end{prop}
Next, we argue that there is a 1-1 correspondence between the clusters.
 \begin{prop}\label{prop:ws}
The function $k \to \hat{k}(k)$ of Proposition \ref{prop:ss} is almost surely injective.
 \end{prop} 
 And we finalise the list of propositions by arguing that the clusters are almost surely equal; the equality of measures is concluded in the proof of the theorem.
 \begin{prop}\label{prop:eq}
For all $k \geq 1$, with $k \to \hat{k}(k)$ as above, we have $C_k = \hat C_{\hat{k}(k)}$ almost surely. Moreover, the function $k \to \hat{k}(k)$ is also almost surely surjective.
 \end{prop} 

We now describe a certain way of exploring excursion clusters using local set processes, then prove the propositions one by one, and finish the section by concluding the proof of the theorem.

\subsection{A local set exploration of excursion clusters}

Throughout the proofs we will make use of the following local set, obtained by exploring the clusters around a line segment until some stopping time. 

\begin{lemma}\label{l:strong_Markov}
Let $\gamma:[0,1]\mapsto \D$ be a simple  curve. We define $\hat\gamma^{exc}(t)$ as the union  of all $\hat C_k $ that intersect $\gamma([0,t])$ and take $\tau$ a stopping time for the filtration $\F_t:=\bigvee_{s\leq t}\sigma(\hat \gamma^{exc}(s))$. We have that $\hat \gamma^{exc}(\tau)$ is also a local set, more precisely $\Phi= \Phi^{\hat \gamma^{exc}(\tau)}+\Phi_{\hat \gamma^{exc}(\tau)}$ where conditionally on $\hat \gamma^{exc}(\tau)$ the law of $\Phi^{\hat \gamma^{exc}(\tau)}$ is a {zero boundary} GFF in $D\backslash \hat \gamma^{exc}(\tau)$ and $\Phi_{\hat \gamma^{exc}(\tau)}=\sum_{k}\sigma_k\nu_k \1_{C_k\cap \eta \neq \emptyset}$ (where again the sum is ordered by descending diameter size of clusters).
\end{lemma}

\begin{proof}
As by the Markov property of the excursion decomposition $(\hat\gamma^{exc}(t))_{t \in [0,1]}$ is a family of increasing local sets, this strong Markov property follows from Lemma 1.3.13 in \cite{Aru} {(in this thesis a continuous process of local sets is considered, but for a right-continuous process like here the proof works equally well).}
\end{proof}

To circumvent some technicalities, we have to tweak this local set further to be able to also explore only a subset of the excursions intersecting the line; this is done so that any non-contractible simple loop around the origin contained in the local set processes, has to be in fact a subset of a single explored non-contractible cluster:

\begin{prop}\label{c.chi_epsilon}
Let $\gamma$ be the straight line segment from $-i$ to $0$ and define $\hat \gamma^{exc}$ to be the closure of the union of all $\hat C_k$ that intersect $\gamma$. For any $\epsilon>0$, there exists a local set $\chi_\epsilon$ that has the following property:
    \begin{enumerate}
        \item It is equal to the {closed} union of certain excursions of $(\hat C_k)_{k \geq 1}$.
        \item It is contained in $\hat \gamma^{exc}$, and contains any cluster $\hat C_k$ that surrounds $0$ with diameter bigger than or equal to $\epsilon$.
        \item if there is a simple loop $\ell \subseteq \chi_\epsilon$ with diameter bigger than $\epsilon$ that surrounds $0$, then there exists $k$ such that $\ell\subseteq \hat C_k$. 
        \item {any simple loop $\ell \subseteq \D$ of diameter at least $\eps$ surrounding the origin has to either be contained in $\chi_\epsilon$ or hit at least two different prime ends of $\D \setminus \chi_\epsilon$.}
    \end{enumerate}
\end{prop}

\begin{proof}[Proof of Proposition \ref{c.chi_epsilon}] We construct the set $\chi_\epsilon$ recursively. We first define the stopping time
\begin{align*}
\hat \tau^0:= \inf\{t\in [0,1]: \exists \delta>0, \exists \mathcal O \text{ c.c. of } B(0,1)\backslash (\hat \gamma^{exc}(t)\cup \gamma\mid_{[-i,-\delta i]}) \text{ with } 0\in O \text{ and } \partial {\D} \cap \overline{\mathcal O}=\emptyset\}.   
\end{align*}
{On the event $\hat \tau^0 = 1$, we just set $\chi_\epsilon= \hat \gamma^{exc}$ and by definition of $\hat \tau^0$ all the conditions hold and there is no cluster surrounding the origin. In fact we will see a posteriori that this event has zero probability and there will be some cluster surrounding $0$ with positive probability, but for now this cannot be excluded.}

{We now work on the event $\hat \tau^0 < 1$. Then} $\hat \gamma^{exc}(\tau^1)\backslash \bigcup_{t<\tau^1}\hat \gamma^{exc}(t)$ is equal to a certain excursion $\hat C_k$. If this excursion surrounds $0$ and has diameter smaller than or equal to $\epsilon$ we finish our exploration and define $\chi_\epsilon= \hat \gamma_0^{exc}(\tau^0) := \hat \gamma^{exc}(\tau^0)$; by definition it satisfies the desired conditions.

If the above is not the case, we will continue our exploration as follows. Note that $\hat \gamma^{exc}(\tau^0)$ is a local set and that the set of excursions $(\hat C_k: C_k\cap \hat \gamma^{exc}(\tau^0)=\emptyset)_{k \geq 1}$ generate an excursion decomposition of the GFF $\Phi^{\hat \gamma^{exc}(\tau^0)}$. Let 
\begin{align*}
    \hat x:=\sup\{\Im(x): x \in \hat \gamma^{exc}(\tau^1)\cap \gamma \}
\end{align*}
where $\Im(x)$ denotes the imaginary part of $x$ and consider the line $\gamma_1$ that goes from $-i\hat x$ to $-\epsilon i$. We now define $\hat \tau^1$ as above, but with $\gamma_{1}$ in the role of $\gamma$ - i.e. in words, if a cluster appears such that $\hat \gamma^{exc}(\tau^0)\cup [-i, -\delta i]$ disconnects the origin from the boundary of the domain, we stop and start exploring from the top-most point of $\gamma^{exc}(\tau^0)$ on the imaginary axis. We again stop if an excursion surrounding $0$ with diameter smaller or equal than $\epsilon$ appears or if $\hat \tau^1 = 1$. Denoting this new bit of exploration by $\hat \gamma_1^{exc}(\tau^1)$, we set $\chi_\epsilon = \hat \gamma^{exc}(\tau^0) \cup \hat \gamma_1^{exc}(\tau^1)$. Otherwise we keep on going. 

{If this procedure finishes at a finite step $j$ we set $\chi_\epsilon = \hat \gamma^{exc}(\tau^0) \bigcup_{i=1}^j \hat \gamma_i^{exc}(\tau^i)$ and otherwise we set $\chi_\epsilon = \hat \gamma^{exc}(\tau^0) \overline{\bigcup_{i\geq1} \hat \gamma_i^{exc}(\tau^i)}$. In the latter case we must have that $\tau_j \to 1$ because the set of clusters is locally finite.}

We now need to prove that these sets satisfy the claimed properties. (1) and (2) are clear from construction, as any cluster surrounding $0$ would need to appear at some $\tau_j$ with $\tau_j < 1$.

To see the point $(3)$, we note that by definition there can be no simple loops $\ell$ surrounding $0$ contained in $ \bigcup_{t<\tau_{j+1}}\hat\gamma^{exc}_j(t) \backslash \hat \gamma^{exc}_{j-1}(\tau_{j})$ and furthermore $\left(\bigcup_{t<\tau_j+1}\hat\gamma^{exc}_j(t) \right) \cap \hat \gamma^{exc}_{j-1}(\tau_{j})$ and $ \left(\bigcup_{t<\tau_{j+1}}\hat \gamma^{exc}_j(t)\right) \cap \hat \gamma^{exc}_{j}(\tau_{j+1})$ both have exactly one point and thus there can not be a simple loop going between $\bigcup_{t<\tau_j+1}\hat \gamma^{exc}_j(t)$ and either $\hat \gamma^{exc}_{j}(\tau_{j+1})$ or $\hat \gamma^{exc}_{j-1}(\tau_{j})$. This proves $(3)$ for the case if the procedure finishes at a finite step $j$ and $\tau_j < 1$, in the other case no such loop $\ell$ exists.

{The last property follows from the fact that by construction $\chi_\eps$ either separates the origin from the boundary using a set of diameter less than $\eps$ or contains the origin.}
\end{proof}

\subsection{Proofs of propositions}
Let us start this subsection by proving Proposition \ref{prop:ss}.
\begin{proof}[Proof of Proposition \ref{prop:ss}] It suffices to prove the proposition for any cluster $C_k$ that surrounds $0$. Let $\gamma$ be the straight line segment from $-i$ to $0$ and define $\hat \gamma^{exc}$ to be the closure of the union of all $\hat C_k$ that intersect $\gamma$ as before and consider the local set $\chi_\epsilon$ from Lemma \ref{c.chi_epsilon}.

We now work recursively starting from the outermost cluster surrounding the origin. In this aim, we construct $\A_{-2\lambda,2\lambda}$, say via SLE$_4(-2)$ like in \cite{ASW}, and we consider $k \in \N$ such that the outer boundary of $C_k$ is in $\A_{-2\lambda,2\lambda}$ and surrounds $0$ (i.e. $C_k$ is the outer-most excursion surrounding $0$); we let $L_k\subseteq \A_{-2\lambda,2\lambda}$ denote this outer boundary. We start by showing that $L_k$ is contained in some $\hat C_{\hat k(k)}$.

{Claim 17 of \cite{ASW} implies that if this level line loop $L_k$ intersects $\D \setminus \chi_\epsilon$, it can touch $\chi_\epsilon$ in at most two of its prime ends of (at the start and at the end of the loop). Furthermore, this can be strengthened. Indeed, using the same proof as that of Claim 17 of \cite{ASW}, one can see that the level line $L_k$ can touch only at one prime end: the same proof implies that after the starting point a small enough initial segment of the loop remains only in the vicinity of one single prime end.} 
{Thus by Proposition \ref{c.chi_epsilon} we conclude that on the event that $L_k$ has diameter larger than $\eps$, it has to be contained in one $\hat C_{\hat k(k)}$. By taking $\epsilon \to 0$, we see that this holds almost surely.}

We now show that the whole cluster $C_k$ is contained in $\hat C_{\hat k(k)}$. To do this, recall from the construction in Section \ref{s:existence} that given the outer boundary, the cluster $C_k$ is constructed by taking $\A_{-2\lambda}$ inside the connected component of ${\D}\backslash L_k$ containing $0$. Further, by uniqueness of FPS, we can use the following iterative recipe to construct $\A_{-2\lambda}$: to obtain $A^k$, we sample $\A_{-\lambda,\lambda}$ inside the connected component of ${\D}\backslash A^{k-1}$ that contains $0$, unless the boundary value is already equal to $0$.  We can now use the argument given in \textit{Uniqueness} in Section 6 of \cite{ASW} to see that each $A^k$ is contained in $\hat \gamma^{exc}$. Indeed, the fact that level lines used to construct $\A_{-\lambda,\lambda}$ do not self-intersect and the fact that if they enter any connected component of ${\D}\backslash \hat \gamma^{exc}$ they cannot touch the boundary of $\hat \gamma^{exc}$ {(Lemma 16 of \cite{ASW})} imply that they cannot enter any connected component of ${\D}\backslash \gamma^{exc}$ at all. This concludes the proof that the outermost cluster $C_k$ surrounding $0$ is contained in $\hat C_{\hat k(k)}$.

To show that the next cluster $C_{k'}$ surrounding $0$ is also contained in some excursion $\hat C_{\hat k(k')}$ it suffices to note that the law of $\Phi$ restricted to the connected component $\mathcal O$ containing $0$ of ${\D}\backslash C_k$ is that of a GFF in ${\D}\backslash C_k$. This implies that the restriction of $(\hat C_k,\hat \mu_k, \hat \sigma_k)_{k \geq 1}$ to $\mathcal O$ is also that of an excursion decomposition of that GFF and so we can repeat the above procedure.

\end{proof}

    Next up is Proposition \ref{prop:es}. The idea is to recover the sign of the cluster by using well-chosen positive test functions whose support is contained in a small neighbourhood of the set.
    \begin{proof}[Proof of Proposition \ref{prop:es}] Similarly to above, it suffices to prove the claim for the outermost cluster $C_k$ surrounding $0$. 
    Let $\gamma:[0,1]\mapsto \overline{B(0,1)}$ be a straight line from $-i$ to $0$ and $\hat C_k$ (or $C_k$) be the outermost cluster surrounding $0$. Define
    \begin{align}\label{e:tau}
    \tau:=\inf\{t\in [0,1]: C_k \in \gamma^{exc}(t)\}.\end{align}
Then $\gamma^{exc}(\tau)$ is a local set by Lemma \ref{l:strong_Markov}. Further, by Proposition \ref{prop:ss}, we know that for all $t \in [0,1]$, it holds that $\gamma^{exc}(t)\subseteq \hat \gamma^{exc}(t)$.
    Thus $\hat \gamma^{exc}(\tau)$ is also a local set and it contains the cluster $\hat C_{\hat k(k))}$. Observe that by construction the closure of $\bigcup_{t< \tau} \gamma^{exc}(t)$ intersects $C_k$ only at $\hat x = \gamma(\tau)$ and that the same holds for $\hat C$.
    
   \begin{claim}\label{c:mesure} Let $A \subseteq \hat C_k\backslash \{\hat x\}$ be a closed set of positive diameter that is measurable w.r.t. $\hat \gamma^{exc}(\tau)$. Consider $(f_n: \D \mapsto [0,1])_{n \geq 1}$, a a family of smooth functions all taking value $1$ on $A$ and equal to $0$ for all points of distance at least $2^{-n}$ of $\hat C_k$. Further, assume that conditionally on $\hat \gamma^{exc}( \tau)$, $f_n$ are independent from the GFF $\Phi^{\hat \gamma^{exc}(\hat \tau)}$. 
   
   Then, almost surely if $\hat \nu_k(A)>0$, $\liminf(\Phi, f_n)\geq \hat \nu_k(A)$ and if $\hat \nu_k(A)<0$, $\liminf(\Phi, f_n)\leq -\hat \nu_k(A)$.
   \end{claim}

Before proving the claim, let us see how it implies the proposition. First, as the claim holds for clusters of any excursion decomposition satisfying the conditions of Theorem \ref{thm:excdcmp}, it holds in particular also for the one constructed via CLE$_4$ and FPS in the previous section, i.e. if we omit all the hats on $C$-s in the statement. 

Further, as $\gamma^{exc}(\tau)$ is a local set contained in $\hat \gamma^{exc}(\tau)$, it is conditionally independent of $\Phi^{\hat\gamma^{exc}(\tau)}$. We can now apply the claim with a closed set $A \subseteq C_k \setminus \{\hat x\}$, and functions $f_n$ chosen depending only on $\hat\gamma^{exc}(\tau)$ and $\Phi_{\hat\gamma(\tau)}$ twice (once for $C_k$, once for $\hat C_{\hat k(k)}$) to obtain the proposition.

It remains to argue the claim.
    \begin{proof}[Proof of Claim \ref{c:mesure}] We can use the local set property of $\hat \gamma^{exc}(\tau)$ to write \begin{align*}
    (\Phi,f_n)=(\Phi_{\hat \gamma^{exc}( \tau)},f_n) + (\Phi^{\hat \gamma^{exc}(\tau)},f_n).\end{align*}
    By conditioning on $\hat \gamma^{exc}( \tau)$ we can see that the variance of $(\Phi^{\hat \gamma^{exc}(\hat \tau)},f_n)$ goes to $0$. We conclude by noting that there exists an $n$ such that the support of $f_n$ does not intersect the closure of $\bigcup_{t<\hat \tau} \hat \gamma^{exc}(t)$.
    \end{proof}
    \end{proof}

We now turn to Proposition \ref{prop:ws} and start with a preliminary lemma.
  \begin{lemma}\label{l.no_inside}
For any connected component $\mathcal O$ of ${\D}\backslash C_k$ that does not contain $\partial {\D}$ we have that 
$\mathcal O \cap \hat C_{\hat k(k)}=\emptyset$. 
In particular, if $\ell_k$ is the outer boundary of $C_k$ and $\mathcal O'$ is the interior, then $\closure{\hat C_{\hat k(k)} \cap \mathcal O'} = C_k$.
   \end{lemma}

    \begin{proof}
We show it for the connected component containing $0$.
        Define $\tau$ as in the proof of Proposition \ref{prop:es}. Note that $\mathcal O$ is also the connected component of ${\D}\backslash \gamma^{exc}(\tau)$ that contains $0$. 
        Because $C_k\subseteq \hat C_{\hat k(k)}$ we have that $\mathcal O \cap \hat \gamma^{exc}(\tau)$ is equal to $\mathcal O \cap \hat C_{\hat k(k)}$. Lemma \ref{lem:BPLS} (2) then implies that $\mathcal O \cap \hat C_{\hat k(k)}$ is a local set of $\Phi^{\gamma^{exc}(\tau)}|_{\mathcal O}$. Conditionally on $\gamma^{exc}$ and the sign of the cluster $C_k$ (WLOG we assume it +1), we see that restricted to $\mathcal O$, $(\Phi^{\gamma^{exc}(\tau)})_{\hat C_{\hat k(k)}}\geq 0$. The first part of Proposition 4.5 of \cite{ALS1} now implies that $\mathcal O \cap \hat C_{\hat k(k)}$ has to be empty.
        
        The second statement follows directly.
    \end{proof}

We can now prove Proposition \ref{prop:ws}, which is maybe the trickiest of the four.

 \begin{proof}[Proof of Proposition \ref{prop:ws}]
 By Proposition \ref{prop:ss} we know that for each $C_k$, there is some $\hat k$ with $C_k \subseteq \hat C_{\hat{k}}$.
    We start by showing that the signs $\sigma_k, \sigma_{k'}$  are independent even when we further condition on the event $E_{k, k'}$ that they do not belong to the same cluster of $(\hat C_k)_{k\geq 1}$, i.e. on the event,
    $$E_{k, k'} = \{\hat k(k) \neq \hat k(k')\}.$$
This is formalized by the following lemma    

\begin{lemma}\label{c:covariance_different_cluster}    We have that
    $$\E\left[(\sigma_{k}\nu_{k}, 1)(\sigma_{k'}\nu_{k'}, 1)1_{E_{k k'}}\right] = 0.$$
    \end{lemma}

Using this lemma, we can argue that the function $k \to \hat{k}(k)$ is injective. Indeed, for any points $z_1, \dots, z_n$ and any nesting levels $j_1, \dots, j_n$, there is some $N$ such that all clusters surrounding these points up to these nesting levels are contained in the first $N$ clusters when ordered by the decreasing size of diameter. We can now write \begin{equation}\label{eq:yes}
\E\left[\left(\sum_{k=1}^N(\sigma_k \nu_k, 1)\right)^2\right] = \sum_{k = 1}^N \E\left[(\nu_k, 1)^2\right].
\end{equation}
Using the lemma we can alternatively write the LHS as:
$$\E\left[\left(\sum_{k=1}^N(\sigma_k \nu_k, 1)\right)^2\right] = \sum_{k = 1}^N \E\left[(\nu_k, 1)^2\right] + \sum_{k \neq k'} \E\left[(\sigma_{k}\nu_{k}, 1)(\sigma_{k'}\nu_{k'}, 1)1_{E^c_{k,k'}}\right].$$
But on the event $E^c_{(z,j),(w,h)}$, we have that $\sigma_{k} = \sigma_{k'}$ and thus all the terms in the second sum are non-negative. But then they have to actually be equal to zero by \eqref{eq:yes}. As this holds for any collection of $z_1, \dots, z_n$ and any nesting heights, and all clusters can be listed this way, we obtain that $k \to \hat{k}(k)$ is injective.

It remains to prove the lemma.

  \begin{proof}[Proof of Lemma \ref{c:covariance_different_cluster}]
In this proof, it will be useful to denote clusters using the point they surround and their level of nesting as follows: for $z\in \D$ and $j\in \N$, we denote $C_{z,j}$ denote respectively the $j-$th outermost cluster that surrounds $z$. So we now fix $z,z' \in \D$ and $j,j'\in \N$ and denote $k=k(z,j)$ and $k'=k(z',j')$ the $k$ and $k'$ such that $C_k=C_{z,j}$ and $C_{k'}=C_{z',j'}$ respectively. We remark that all clusters $C_k$ can be listed by considering only dyadic $z \in \D$. 
    
    We first prove the lemma when both clusters are outer-most, i.e. when $j=j'=1$. Consider the local set $\hat \gamma^{exc}(\tau_z \wedge \tau_{z'})$ along a line segment $\gamma$ from the boundary to $z$ and then to $w$, stopped at time $\tau_z\wedge \tau_w$,  when either a cluster of the decomposition $(\hat C_k)_{k \geq 1}$ around $z$ or $w$ appears. Assume, WLOG that it is $\hat C_{z,1}$ that appears. In that case, and on the event  $E_{k, k'}$, we have that $C_{z',1}\cap \hat C_{z,1}=\emptyset$. Denote by $\mathcal O'$ the connected component of $\D\backslash\hat \gamma^{exc}(\tau_z \wedge \tau_{z'})$ that contains $z'$. Further, we define $\mathcal O$ as follows. 
    \begin{enumerate}
        \item Either $C_{z,1}\subseteq \hat C_{z,1} $, in which case we define $\mathcal O=\emptyset$; 
        \item or $C_{z,1}\cap \hat C_{z,1} =\emptyset$, in which case we define $\mathcal O$ as the connected component of $\D\backslash\hat \gamma^{exc}(\tau_z \wedge \tau_{z'})$ that contains $z$.
    \end{enumerate}
 Note that in both cases $\mathcal O \cap \mathcal O'=\emptyset$. We further claim the following. 
 
 \begin{claim}
 Let $A$ be the closure of the union of the outer-most boundaries of the outer-most clusters of $(C_k)_{k\geq 1}$ that are contained in either $\mathcal O$ or $\mathcal O'$.  Then a.s. $A$ restricted to both $\mathcal O$ or $\mathcal O'$ is equal to $\A_{-2\lambda,2\lambda}$  of $\Phi^{\hat \gamma^{exc}(\tau_z \wedge \tau_{z'})}$ restricted to $\mathcal O$ or $\mathcal O'$, respectively.
 \end{claim}
   
 \begin{proof}[Proof of the claim.]
  {
  The set $A\cup \hat \gamma^{exc}(\tau_z \wedge \tau_{z'})$ is a local set that, restricted to $\mathcal O \cup \mathcal O'$ is a.s. equal to $\A_{-2\lambda,2\lambda} \cup \hat \gamma^{exc}(\tau_z \wedge \tau_{z'})$. This is because by construction in Section \ref{s:existence} the collection outer-most boundaries of the outermost loops of $C_k$ is equal to the collection of loops of $\A_{-2\lambda,2\lambda}$.}
  
 { Further, by Lemma \ref{lem:BPLS}, have that for any $x\in \mathcal O' \cap D\backslash (\A_{-2\lambda,2\lambda}\cup \hat \gamma^{exc}(\tau_z \wedge \tau_{z'}) $, the harmonic function $\Phi_{\hat \gamma^{exc}(\tau_z \wedge \tau_{z'})}$ is equal to $\pm 2\lambda$. Thus by (2) Lemma \ref{lem:BPLS}, we see that $A$ is a local set of $\Phi^{\hat \gamma^{exc}(\tau_z \wedge \tau_{z'})}$ restricted to $\mathcal O'$ is thin (because as a subset of $\A_{-2\lambda, 2\lambda}$ of $\Phi$, its Minkowski dimension is smaller than 2) and its harmonic function is equal to $\pm 2\lambda$. We conclude using the uniqueness of TVS, Theorem \ref{thm:tvs}, that it is equal to $\A_{-2\lambda, 2\lambda}$ of $\Phi^{\hat \gamma^{exc}(\tau_z \wedge \tau_{z'})}$ restricted to $\mathcal O'$. The same argument holds for $\mathcal O$, in case it is nonempty.}
 \end{proof}

    We now notice that $A$ is measurable w.r.t. $\Phi^{\hat \gamma^{exc}(\tau_z \wedge \tau_{z'})}$, and  that conditionally on $\hat \gamma^{exc}(\tau_z \wedge \tau_{z'})$, $\Phi^{\hat \gamma^{exc}(\tau_z \wedge \tau_{z'})}$ is independent of $\Phi_{\hat \gamma^{exc}(\tau_z \wedge \tau_{z'})}$. So now, if we are on the case (1)
    \begin{align*}
    &\E\left[\sigma_{k} \sigma_{k'} (\nu_k,1) (\nu_{k'},1) \1_{E_{k,k'}}\1_{(1)}\right]\\ 
    &=\E\left[\sigma_{k} (\nu_k,1)\E\left[\sigma_{k'}(\nu_{k'},1) \mid  \hat \gamma^{exc}(\tau_z \wedge \tau_{z'}),\Phi_{\hat \gamma^{exc}(\tau_z \wedge \tau_{z'})}\right]\1_{E_{k,k'}}\1_{(1)} \right]=0,\end{align*}{
    where we use that the event (1) and $E_{k,k'}$ are measurable with respect to $\gamma^{exc}(\tau_z \wedge \tau_{z'}),\Phi_{\Phi^{\hat \gamma^{exc}(\tau_z \wedge \tau_{z'})}}$ (recall  that $ \gamma^{exc}(\tau_z \wedge \tau_{z'}) \subseteq \hat \gamma^{exc}(\tau_z \wedge \tau_{z'})$), and the previous claim together with the construction of the decomposition $(C_k, \sigma_k, \nu_k)_{k \geq 1}$ in Section \ref{s:existence}. }
    
    For the case (2), a similar computation is needed:
    \begin{align*}
    &\E\left[\sigma_{k} \sigma_{k'} (\nu_k,1) (\nu_{k'},1) \1_{E_{k,k'}}\1_{(2)}\right]\\ 
    &=\E\left[\E\left[\sigma_{k'}(\nu_{k'},1) \mid  \hat \gamma^{exc}(\tau_z \wedge \tau_{z'}),\Phi_{\hat \gamma^{exc}(\tau_z \wedge \tau_{z'})}\right]\E\left[\sigma_{k}(\nu_{k},1) \mid  \hat \gamma^{exc}(\tau_z \wedge \tau_{z'}),\Phi_{\hat \gamma^{exc}(\tau_z \wedge \tau_{z'})}\right]\1_{E_{k,k'}}\1_{(2)} \right]\\ 
    &=0.\end{align*}
    {Here we used the claim above together with the fact that $\Phi^{\hat \gamma^{exc}(\tau_z \wedge \tau_{z'})}$ restricted to disjoint components $\mathcal O, \mathcal O'$ are independent.}

For clusters at further levels we discover $\hat \gamma^{exc}(\tau_z \vee \tau_{z'})$, i.e. we wait until both outermost clusters appear, and then iterate inside the connected components of complement of $\hat \gamma^{exc}(\tau_2)$ containing either $z$ or $z'$ as above.\end{proof}
 \end{proof}
 
Finally, Proposition \ref{prop:eq} follows by some further considerations on local sets and by also using the known Minkowski dimension of $\A_{-2\lambda, 2\lambda}$.
 \begin{proof}[Proof of Proposition \ref{prop:eq}]
    We start by showing that 
$C_k = \hat C_{\hat{k}(k)}$ almost surely. By Proposition \ref{prop:ss}, we have that $C_k \subseteq \hat C_{\hat{k}(k)}$. Further by Lemma \ref{l.no_inside}, we know that $C_{\hat{k}(k)} \setminus C_k$ can only intersect the connected component $\mathcal O$ of $D \setminus C_k$ that contains $\partial \D$ on its boundary. 

Like in the proof of Claim \ref{c:covariance_different_cluster}, all clusters can be listed from outermost towards the interior around dyadic points $z_k$. Thus it suffices to prove $C_k = \hat C_{\hat{k}(k)}$ for outermost clusters and this in turn follows from showing the following claim: for any curve $\gamma$ along the dyadics starting from $\partial \D$ to some $z_k$, we have that the complements $\mathcal O, \hat {\mathcal O}$ of $A= \gamma^{exc}(1)$ and $\hat A =\hat \gamma^{exc}(1)$, which  share boundary with $\D$, agree almost surely. To see this observe that both $A, \hat A$ are local sets, and by Lemma \ref{lem:BPLS} (2),  $\hat A \setminus A $ is also a local set of $\phi^{\gamma^{exc}(1)}$ restricted to $\mathcal O$. Further, as $C_k \subseteq C_{\hat{k}(k)}$ and by Proposition \ref{prop:ws} the function $k \to \hat{k}(k)$ is injective, we see that $\hat A \setminus A \subseteq \A_{-2\lambda, 2\lambda}(\D)$. Hence $\hat A \setminus A$ has Minkowski dimension strictly less than 2 \cite{ASW} and it is a thin local set of $\phi^{\mathcal O}$, connected to boundary with zero boundary values. Thus by Lemma 9 of \cite{ASW} it is almost surely empty.

It remains to show that the function $k \to \hat k$ is surjective. This follows from a very similar argument as above. Indeed, observe that any cluster $\hat C_{k^*}$ that is not equal to some $C_k$ is contained in the closed union of outer boundaries of $C_k$ in some finite iteration step of the construction of the excursion decomposition $(C_k, \nu_k, s_k)_{k \geq 1}$. But these outer boundaries are given by independent copies of $\A_{-2\lambda, 2\lambda}$. Thus we can repeat the argument above to obtain that such clusters of positive diameter do not exist. To see that there are no clusters whose support is just a point, we recall that almost surely the 2D GFF does not put any mass on single points.

 \end{proof}

 \subsection{Conclusion of the uniqueness of the excursion decomposition}
    
    \begin{proof}[Proof of uniqueness in Theorem \ref{thm:excdcmp}] Take $(\sigma_k,\nu_k,C_k)$ and $(\hat \sigma_k, \hat \nu_k, \hat C_k)$ two decompositions, where the first one is the one constructed in the previous section. 
   From Proposition \ref{prop:ss} we know that for every $k$ there is $\hat k(k)$ such that $C_k= \hat C_{\hat k(k)}$ and by Propositions \ref{prop:ws}, \ref{prop:eq} this assignment is bijective. We also know that the signs agree, so it suffices to show that the measures $\nu_k$, $\hat \nu_k$ agree.
   
 To see this observe that for any simple curve $\gamma$ from the boundary and for any time $t$ a.s., we have that $\gamma^{exc}(t) = \hat \gamma^{exc}(t)$. But now notice that by the Markov decomposition, we can conclude that almost surely for any curve $\gamma$ along the dyadics and any rational time in $H^{-1-\eps}(\D)$
   $$\phi_{\gamma^{exc}(t)} = \hat \phi_{\gamma^{exc}(t)}.$$
   But the local set process $\gamma^{exc}(t)$ is right-continuous, and for any decreasing sets $D_n$ with $\bigcap D_n = D$, $\Phi^{D_n}$ is also continuous. We conclude that in fact for all times $t \in [0,1]$, it holds that 
      $$\phi_{\gamma^{exc}(t)} = \hat \phi_{\gamma^{exc}(t)}$$
      and in particular it holds at the appearance of any cluster of diameter at least $\epsilon$. This concludes that in fact for all $k \geq 1$, we have $\nu_k = \hat \nu_{\hat{k}(k)}$ and the theorem follows.

    \end{proof}
    
    \section{Convergence from the metric graph}\label{s:convergence}
In this section, we prove the convergence of the excursion decomposition of the metric graph GFF to that of the continuum GFF. We will work in the same set-up as in Section 4.1 of \cite{ALS2}, except that the domain $D$ will always be simply connected. 

For all $n \geq 1$, let $\tilde \phi_n$ be a metric GFF in a bounded graph $D_n\subseteq (2^{-n}\Z)^2$. We define $(\widetilde C_k^{(n)},
\tilde\sigma^{(n)}_k,\tilde\nu_k^{(n)})_{k \geq 1}$ as the sequence of sign clusters of $\tilde \phi_n$, the respective signs and sign excursions, ordered by decreasing size of cluster diameter. Here, by a sign excursion we mean the absolute value of the restriction of the GFF to the cluster 
$\widetilde C_k^{(n)}$, i.e.
\begin{align*}
    \tilde\nu_k^{(n)}(dx) = \tilde\sigma_k^{(n)}\tilde \phi_n (x) \1_{x\in \widetilde C_k^{(n)}} dx. 
\end{align*}

We now take a sequence of (metric) graphs $\widetilde D_n\subseteq (2^{-n}\Z)^2$ converging to a bounded and simply connected domain $D\subseteq \C$ in the sense that {their complements inside some large box $[-C, C]^2 \supseteq D$ converge in the Hausdorff topology (as in Section 4.1.1 of \cite{ALS2})}. 

The main result of this section is the following.
\begin{thm}[Convergence of the excursion decomposition]\label{thm:conv}
Let $\tilde \phi_n$ be a sequence of zero boundary metric graph GFFs on $\widetilde D_n$ that are coupled with a GFF $\Phi$ such that a.s. 
$\tilde \phi_n \to \Phi$ in, say, $H^{-\epsilon}$.

Then for every $k>0$,  $ \widetilde C_k^{(n)}\to C_k$, 
$\tilde\nu_k^{(n)}\to \nu_k$ and $\tilde\sigma_k^{(n)}\to \sigma$ as $n\to \infty$, where the convergence is in probability and in the Hausdorff topology for the first, and in the weak topology of measures for the first and second components respectively.
\end{thm}

In large lines, one could say that the theorem follows by patching together different convergence results for each element, all of which are already present in the literature. This patching, however, does require some care, mainly to rule out different possible spurious contributions from microscopic clusters. Notice that we will not use the uniqueness claim of the theorem to identify the limit; rather we will identify excursion clusters, signs and measures one by one.

We start from a lemma that ensures the tightness of the sequences of measures
$(\tilde{\nu}^{(n)}_{k})_{n\geq 0}$ and allows us to see that no spurious extra mass is produced in the limit by infinitesimal excursion clusters.

\begin{lemma}\label{lem:L2L4}
Let $(\widetilde C_k^{(n)},\tilde\sigma^{(n)}_k,\tilde\nu_k^{(n)})$ be an excursion decomposition of the metric graph GFF $\tilde \phi_n$ and let $J$ be any (deterministic) index set. Then for any $q \in \N$
\begin{equation}\label{eq:L4ineq}
\E\left[(\tilde \phi_n, f)^{2q}\right] \geq \sum_{k \in J} \E\left[(\tilde \nu_k^{(n)}, f)^{2q} \right]+ \E \left[(\1_{\widetilde D_n \backslash \cup_{k \in J}\widetilde C_k^{(n)}}\tilde \phi_n, f)^{2q}\right].
\end{equation}
and
\begin{equation}\label{eq:L2eq}
\E\left[(\tilde \phi_n, f)^2\right] = \sum_{k \in J} \E\left[(\tilde \nu_k^{(n)}, f)^2 \right]+ \E \left[(\1_{\widetilde D_n \setminus \cup_{k \in J}\widetilde C_k^{(n)}}\tilde \phi_n, f)^2\right]
\end{equation}
\end{lemma}

\begin{proof}
For the inequality \eqref{eq:L4ineq}, it suffices to prove it for any finite index set and any $q \geq 1$, then the case of infinite index sets follows by dominated convergence.
We decompose
\begin{displaymath}
    \tilde \phi_n =
    \sum_{k \in J} \tilde\sigma_k^{(n)} \tilde{\nu}^{(n)}_{k}
    + \1_{\widetilde D_n \setminus \cup_{k \in J}\widetilde C_k^{(n)}}\tilde \phi_n .
\end{displaymath}
Then, we write $\mathbb{E}[(\tilde \phi_n,f)^{2q}]$ as the sum of three types of terms, the first being
\begin{displaymath}
 \sum_{k \in J}
    \mathbb{E}[(\tilde{\nu}^{(n)}_{k},f)^{2q}]
    +
    \mathbb{E}[(\1_{\widetilde D_n \setminus \cup_{k \in J} \widetilde C_k^{(n)}}
    \tilde \phi_n,f)^{2q}],
\end{displaymath}
the second type of terms are a binomial coefficients times
\begin{displaymath}
    \mathbb{E}[(\tilde\sigma_k^{(n)}\tilde{\nu}^{(n)}_{k},f)^{p}(\1_{\widetilde D_n \setminus \cup_{k \in J} \widetilde C_k^{(n)}}
    \tilde \phi_n,f)^{2q -p}]
\end{displaymath}
and the last type of terms are constant times
\begin{displaymath}\mathbb{E}[(\tilde\sigma_k^{(n)}\tilde{\nu}^{(n)}_{k},f)^p(\tilde\sigma_j^{(n)}\tilde{\nu}^{(n)}_{j},f)^{2q-p}],
\end{displaymath}
with $k \neq j$. 

Now, when $p$ is even, we can lower bound the second and third types of terms by $0$. However, we claim that when $p$ is odd, they are equally zero by sign symmetry. Indeed, conditionally on 
$(\widetilde C_k^{(n)},\tilde\sigma_k^{(n)},\tilde{\nu}^{(n)}_{k})_{k \in J}$,
the field 
$\1_{\widetilde D_n \setminus \cup_{k \in J}\widetilde C_k^{(n)}} \tilde \phi_n$
has the same distribution as its {additive inverse}.
Thus for $p$ odd
\begin{displaymath}
  \sum_{k \in J}\mathbb{E}[(\tilde\sigma_k^{(n)}\tilde{\nu}^{(n)}_{k},f)^p(\1_{\widetilde D_n \setminus \cup_{k \in J} \widetilde C_k^{(n)}}
    \tilde \phi_n,f)^{2q-p}] = 0.
\end{displaymath}
But also all the signs $\tilde \sigma_k^{(n)}$ are i.i.d. Rademacher random variables, thus also for all $j \neq k \in J$
    \begin{displaymath}
\mathbb{E}[(\tilde\sigma_k^{(n)}\tilde{\nu}^{(n)}_{k},f)^p(\tilde\sigma_j^{(n)}\tilde{\nu}^{(n)}_{j},f)^{2q-p}] = 0
\end{displaymath}
and we conclude the first part. The second part for finite index sets follows from the computation above, as in the case $q = 1$, there are no cross-terms with even exponents; for the infinite sums, we can use dominated convergence, guaranteed by say the case $q = 2$ in \eqref{eq:L4ineq}.
\end{proof}

\begin{proof}[Proof of Theorem \ref{thm:conv}]We start by noting that thanks to the uniqueness of the excursion decomposition and Lemma 4.10 of \cite{ALS2}, we only need to prove convergence in law of 
$(\tilde \phi_n, (\widetilde C_k^{(n)},\tilde\sigma^{(n)}_k,\tilde\nu_k^{(n)})_k)$ as $n\to \infty$. 

Now, $\tilde \phi_n$ are tight by assumption; for any $k$, 
$\widetilde C_k^{(n)}$ are tight as random closed sets in a compact domain, $\tilde\sigma^{(n)}_k$ are tight as $\pm 1$ valued random variables and finally $\tilde\nu_k^{(n)}$ are tight by the first equality in Lemma \ref{lem:L2L4}.
Thus, using Tychonoff theorem, we see that $(\tilde \phi_n, (\widetilde C_k^{(n)},\tilde\sigma^{(n)}_k,\tilde\nu_k^{(n)})_k)$ is tight, and thus we can find a subsequence of it (we denote it the same way) and use Skorokhod's representation theorem to obtain the almost sure convergence \begin{align*}
(\tilde \phi_n, (\widetilde C_k^{(n)},\tilde\sigma^{(n)}_k,\tilde\nu_k^{(n)})_k) \to (\hat \Phi, (\hat C_k,\hat \sigma_k, \hat \nu_k)_k).\end{align*}
We just need to identify $(\hat \Phi, (\hat C_k,\hat \sigma_k, \hat \nu_k)_k)$ as the elements of the excursion decomposition. First, it is clear that $(\hat \sigma_k)_k$ are i.i.d. Radamacher random variables and $\hat \Phi$ is a GFF in $D$.

First, note that if we only study the outer most clusters 
{$\mathring C_{\mathbf k}^{(n)}:=\hat C_{k({\mathbf k})}^{(n)}$} (i.e. those that are not surrounded by any other cluster), then the outer boundaries of those outermost clusters converge to the loops of CLE$_4$ in the sense that the outer boundaries of the $m$ largest outermost discrete clusters converge to outer boundaries $m$ largest continuum ones, and moreover the closed union of all outermost cluster boundaries converges to CLE$_4$ - these statements follow from the work in \cite{LupuConvCLE}. More precisely, the main statement of that paper does not directly apply these claims  - it does not exclude long thin filament-like clusters with limits in the interior of CLE casket; however with further work it can be deduced with the same methods; see e.g. Lemma 4.13 in \cite{ALS2} for a context, where similar care is needed, or proof of Lemma 6 of \cite{QW15}.

Now, notice that once we manage to identify the outermost clusters, their signs and measures, then we can recursively continue. Indeed, as the closure of the union of outermost clusters $\mathring C^{(n)}_{\mathbf k}$ is a local set for all $n$, we conclude that in the limit, conditionally on the closure of the union of all outermost clusters $\mathring C_{\mathbf k}$, the law of $\hat \Phi$ restricted to $D \setminus (\cup_{\mathbf k} \mathring C_{\mathbf k})$ is that of a zero boundary GFF in $D \setminus (\cup_{\mathbf k} \mathring C_{\mathbf k})$. Thus we see that once we can deal with outermost excursions, the convergence will also hold for excursions that are surrounded by finitely many excursions. As for any $\epsilon>0$ the number of excursions of diameter bigger than $\epsilon$ is almost surely finite, we have reduced the proposition to proving convergence for outermost clusters. This convergence is the content of the following claim.

\begin{claim}\label{c:conv_FPS}
    Fix any loop $\ell$ of $\A_{-2\lambda,2\lambda}$ and consider the sequence of clusters $\mathring C_{\mathbf k}^{(n)}$ whose outer boundaries converge to $\ell$. Then $\mathring C_{\mathbf k}^{(n)}$ converges to the union of $\ell$ with the FPS $\A_{\pm 2\lambda}$ of the GFF $\Phi^{\A_{-2\lambda,2\lambda}}$ restricted to $\mathcal O$, the interior of $\ell$. Furthermore, $\mathring \nu_{\mathbf k}^{(n)}$ converges to the measure $\nu_{\A_{\pm 2\lambda}}$ associated to this FPS.
\end{claim}

\begin{proof}[Proof of Claim \ref{c:conv_FPS}]
    First, we note that the union of all outermost clusters $\mathring C_{\mathbf k}^{(n)}$ is a local set of $\tilde \phi_n$, thus its limit 
    {$\mathring C_{\mathbf k}$} is a local set of $\hat \Phi$ by Lemma \ref{lem:BPLS}. Further, when restricted to the interior $\mathcal O$ of the outermost boundary $\ell$ of the cluster $\mathring C_{\mathbf k}$, this limit is a local set that satisfies the properties of an FPS of level $\pm 2\lambda$ (for the GFF $\Phi^{\A_{-2\lambda,2\lambda}}$ restricted to $\mathcal O$). Thus it is equal to this set by the uniqueness of FPS, Theorem \ref{Thm::FPS}. In particular, this means that {$\mathring C_{\mathbf k}$}, 
    the limit of $ \mathring C_{\mathbf k}^{(n)}$, is equal to some outermost cluster $C_{k(\mathbf k)}$.

   To identify the limiting excursion measures, we will follow a strategy similar to what was used in Section \ref{S:uniqueness} to deduce the equality of excursion clusters and excursion measures by a no extra mass argument. Additional convergence issues are taken care by Lemma \ref{lem:L2L4}. Let us flesh it out here.

   First, as no other subsequential limit of an outermost cluster 
    $\mathring C_{\mathbf k}$ can intersect $\mathcal O$, we conclude from the Markov decomposition w.r.t. the FPS $\A_{\pm 2\lambda}$ in $\mathcal O$ that \begin{equation}\label{eq:oneside}
    \mathring \nu_{\mathbf k}:=\lim_{n\to +\infty} \nu_{\mathbf k}^{(n)}\geq \nu_{k(\mathbf k)}\end{equation} in terms of positive measures.

    Further, one needs to show that there is no extra mass on the subsequential limiting clusters, that is possibly compensated by some infinitesimal excursions in the limit.

    To see this, recall that the closure of the union of outermost clusters $\mathring C_{\mathbf k}$ forms a local set. Also, from the argument above we see that this local set is equal to the local set obtained by taking CLE$_4$ and first passage sets of height $\pm 2\lambda$ inside each of the cluster, i.e. it is equal to the closed union of outermost clusters $\mathring C_{\mathbf k}$. Let us denote this set by $A$.
    
    By Lemma \ref{lem:L2L4}, we can write the sum over outermost excursion clusters $\mathring C_{\mathbf k}$, 
    $$\E\left[(\Phi,1)^2)\right] = \sum_{\mathbf k} 
    \E\left[ (\nu_{k(\mathbf k)},1)^2\right] + \E\left[(\Phi^{A},1)^2\right].$$
    On the other hand, by the claim above and the first point of Lemma \ref{lem:L2L4}, dominated convergence gives us that for any $\mathbf k$ 
    $$\E\left[(\mathring\nu^{(n)}_{\mathbf k},1)^2\right] \to 
    \E\left[(\mathring\nu_{\mathbf k},1)^2\right].$$ Similarly, by Corollary 4.5 from \cite{ALS2} and part 1 of Lemma \ref{lem:L2L4} we also have that $\E\left[(\tilde\phi_{n}^{A^{(n)}},1)^2\right] \to \E\left[(\Phi^{A},1)^2\right]$, where $A^{(n)}$ denotes the local set given by the closed union of $\mathring C_{\mathbf k}^{(n)}$. Thus we also have 
    $$\E\left[(\Phi,1)^2)\right] = \sum_{\mathbf k} \E\left[ (\mathring \nu_{\mathbf k},1)^2\right] + \E\left[(\Phi^{A},1)^2\right].$$
    But now recall that $\mathring \nu_{\mathbf k} \geq\nu_{k(\mathbf k)}$, from where we see that in fact we have to have a one to one correspondence between clusters with positive measure, and the equality has to hold in \eqref{eq:oneside}.

\end{proof}

\end{proof}

    \section{Further comments and conjectures: crossing probabilities, 2D continuum GFF as a limit of spin models and the excursion decomposition of the DGFF}\label{s:further}

    \subsection{Uniform continuity of crossing probabilities by sign excursions}
     First we deal with the continuity, uniformly with respect to the scale, of annuli crossing probabilities by metric graph excursion clusters. \footnote{We hereby send greetings to Jian Ding who asked us this question.} This confirms the assumption in Remark 2 of \cite{DingWirth}. 
     A similar statement for first passage sets was proved in Corollary 5.1 in \cite{ALS2}.
    
The set-up is as follows. Consider $D=(-1,1)^{2}$, and let $\widetilde{D}_{n}$ be the metric graph approximation of $D$ in the square lattice $\frac{1}{n}\Z^{2}$.
    Let $\tilde{\phi}_{n}$ be the metric graph GFF on $\widetilde{D}_{n}$ with $0$ boundary conditions.
    For $a\in (0,1)$, let $S_{a}$ denote the square contour
    $S_{a}=\partial((-a,a)^{2})$
    For $a\leq b\in (0,1)$, let $p_{n}(a,b)$ the probability that there is a sign cluster of $\tilde{\phi}_{n}$ that crosses from $S_{a}$ to $S_{b}$, i.e. crosses the annulus $[-b,b]^2 \setminus (-a,a)^2$.
    
    Note that $p_{n}(a,a)=1$ and that for fixed $n$, the function $p_{n}(a,b)$ is continuous on
    $\{a,b\in (0,1)^{2} : a\leq b\}$. Indeed, this is due to the fact that for $a\in (0,1)$ fixed, a.s., if a cluster of $\tilde{\phi}_{n}$ intersects
    $S_{a}$, then it also intersects $(-a,a)^{2}$ and 
    $D\setminus [-a,a]^{2}$.
    However, we are interested in the continuity of 
    $(a,b)\mapsto p_{n}(a,b)$ uniformly in $n\geq 1$. We will deduce this from the convergence of excursion clusters and continuity in the continuum. 
    
    Denote by $p(a,b)$ further the probability that there is an excursion cluster of the continuum GFF $\Phi$ on $D$ that crosses from $S_{a}$ to $S_{b}$ and observe that again, $p(a,a)=1$.

    \begin{prop} [Uniform continuity of crossing probabilities]\label{prop:perco}
        The following holds.
        \begin{enumerate}
            \item The function $p(a,b)$ is continuous on
            $\{a,b\in (0,1)^{2} : a\leq b\}$.
            \item The sequence $(p_{n}(a,b))_{n\geq 1}$ converges to
            $(p(a,b))_{n\geq 1}$ uniformly on compact subsets of 
            $\{a,b\in (0,1)^{2} : a\leq b\}$.
            \item In particular, the functions 
            $(a,b)\mapsto p_{n}(a,b)$ are continuous uniformly in $n\geq 1$ on any compact subset of $(0,1)^2$.
        \end{enumerate}
    \end{prop}
\begin{rem}        Note that $p(a,b)$ is not continuous on 
    $\{a,b\in [0,1]^{2} : a\leq b\}$, that is if one allows $a=0$ or $b=1$.
    For instance, 
    \begin{displaymath}
        \lim_{a\to 1} p(a,1) = 0 \neq \lim_{a\to 1} p(a,a) = 1.
    \end{displaymath}
\end{rem}

We now prove the proposition. 
    \begin{proof}[Proof of Proposition \ref{prop:perco}]
        Point (3) is a direct consequence of (1) and (2); and point (2) follows from the convergence of clusters; see Theorem \ref{thm:conv}.
        
Finally,
        (1) follows by combining the following 3 facts.
        \begin{itemize}
            \item First, given a fixed $a\in (0,1)$, a.s. there exists an excursion cluster of $\Phi$ intersecting $S_{a}$.
        This follows e.g. from the fact that a.s., there is a Brownian loop in the Brownian loop soup that intersects $S_{a}$.
        \item 
        Second, given a fixed $a\in (0,1)$, a.s., if an excursion cluster intersects
        $S_{a}$, then it also intersects $(-a,a)^{2}$ and 
        $D\setminus [-a,a]^{2}$, that is to say no excursion cluster is tangential to $S_{a}$.
        This is Lemma \ref{lem non tangency} below.
        \item         Third, for every $\varepsilon>0$, there are a.s. finitely many excursion sets of diameter larger than $\varepsilon$ - this is the local finiteness conditions in Theorem \ref{thm:excdcmp}.
        \end{itemize}

    \end{proof}

It remains to state and prove the above-mentioned Lemma \ref{lem non tangency}.
    \begin{lemma}[Non-tangency of excursion clusters]
    \label{lem non tangency}
    Fix $a\in (0,1)$. Then
    \begin{displaymath}
        \P(\exists\,C \text{ excursion cluster of } \Phi,
        C\cap S_{a}\neq\emptyset, \text{ but }
        C\cap(-a,a)^{2} = \emptyset \text{ or }
        C\cap(D\setminus [-a,a]^{2})=\emptyset)
        =0.
    \end{displaymath}
    \end{lemma}
    
    \begin{proof}
        Denote
        \begin{eqnarray*}
            p_{\rm ext}(a) &=&
            \P(\exists\,C \text{ excursion cluster of } \Phi,
            C\cap S_{a}\neq\emptyset, 
            C\cap(-a,a)^{2} = \emptyset )
            ,
            \\
            p_{\rm int}(a) &=&
            \P(\exists\,C \text{ excursion cluster of } \Phi,
            C\cap S_{a}\neq\emptyset, 
            C\cap(D\setminus [-a,a]^{2}) = \emptyset )
            .
        \end{eqnarray*}
        We claim that there can be at most countably many $a\in (0,1)$ such that
        $p_{\rm ext}(a)>0$ or $p_{\rm int}(a)>0$.
        Indeed, let be $(C_{k})_{k\geq 1}$ an enumeration of excursion clusters, for instance by decreasing diameter.
        Denote
        \begin{displaymath}
            \underline{a}_{k} = \sup\{a\in (0,1) : C_{k}\cap S_{a}\neq \emptyset\},
            \qquad
            \overline{a}_{k} = \inf\{a\in (0,1) : C_{k}\cap S_{a}\neq \emptyset\}.
        \end{displaymath}
        If $p_{\rm ext}(a)>0$, resp. $p_{\rm int}(a)>0$,
        then $a$ is an atom for the distribution of
        $\underline{a}_{k}$, resp. $\overline{a}_{k}$,
        for at least one of the $k\geq 1$.
        And the number of atoms of a probability distribution is at most countable.

        We further claim that the functions 
        $a\mapsto p_{\rm ext}(a)$ and
        $a\mapsto p_{\rm int}(a)$
        are both non-decreasing.
        Therefore, being empty is the only way for the sets $p_{\rm ext}^{-1}((0,1])$ and $p_{\rm int}^{-1}((0,1])$ to be at most countable.

        First, let us explain the monotonicity of $a\mapsto p_{\rm int}(a)$.
        Let $a<b\in (0,1)$.
        Let $\Phi_{2}$ be the continuum GFF on the smaller square
        $(-a/b, a/b)^{2}$, with $0$ boundary conditions.
        The excursion sets of $\Phi$ and $\Phi_{2}$ are naturally coupled,
        by using the same Brownian loop soup in $D$.
        The restriction of this Brownian loop soup to $(-a/b, a/b)^{2}$
        is a Brownian loop soup in $(-a/b, a/b)^{2}$. 
        In this coupling, an excursion set of $\Phi$ that is contained in
        $(-a/b, a/b)^{2}$ is also an excursion set of $\Phi_{2}$.
        Therefore,
        \begin{displaymath}
            p_{\rm int}(a) \leq 
            \P(\exists\,C \text{ excursion cluster of } \Phi_{2},
            C\cap S_{a}\neq\emptyset, 
            C\cap((-a/b, a/b)^{2}\setminus [-a,a]^{2}) = \emptyset ).
        \end{displaymath}
        But, by scaling, the right-hand side above equals $p_{\rm int}(b)$.

        The monotonicity of $a\mapsto p_{\rm ext}(a)$ is only slightly more complicated.
        Let $a<b\in (0,1)$.
        Let $\Phi_{3}$ be the continuum GFF on the larger square
        $(-b/a, b/a)^{2}$, with $0$ boundary conditions.
        We couple the excursion clusters of $\Phi$ and $\Phi_{3}$
        by using the same Brownian loop soup on $(-b/a, b/a)^{2}$.
        If $C$ is an excursion set of $\Phi_{3}$ such that
        $C\cap S_{b}\neq \emptyset$ and $C\cap (-b,b)^{2}=\emptyset$,
        then by removing the Brownian loops intersecting
        $(-b/a, b/a)^{2}\setminus D$, 
        the cluster $C$ splits into countably many clusters
        $(C'_{j})_{j\geq 1}$, each of the $C'_{j}$ being an excursion cluster for $\Phi$.
        Since no Brownian loop is tangent to $S_{b}$, and by local finiteness of the excursion clusters of $\Phi$,
        at least one of the $C'_{j}$ has to intersect $S_{b}$,
        and therefore is also tangent to $S_{b}$. Thus,
        \begin{displaymath}
            p_{\rm ext}(b) \geq 
            \P(\exists\,C \text{ excursion cluster of } \Phi_{3},
            C\cap S_{b}\neq\emptyset, 
            C\cap[-b,b]^{2} = \emptyset ).
        \end{displaymath}
        By scaling, the right-hand side above equals $p_{\rm ext}(a)$.
    \end{proof}

   \subsection{2D continuum GFF as a limit of spin models and the corresponding FK representation}
    
    Recall the correspondence between the Ising model and its FK representation:

    \begin{thm}[FK representation of the Ising model]
    Let $(\sigma_v)_{v \in V}$ be the free boundary Ising model on a graph $G = (V, E)$ with inverse temperature $\beta$ and edge-weights $(J_e)_{e \in E}$. 
    
    Then its FK representation is the bond percolation configuration $(\omega_e)_{e \in E}$ on $G$, defined on the same probability space, that satisfies the following properties with $p_e = 1-e^{-2\beta J_e}$:
    \begin{enumerate}
    \item Its marginal law is given by a bond percolation with the law: 
    $$\P(\omega_e) \propto \Pi_{e \in E}p_e^{\omega_e}(1-p_e)^{1-\omega_e}2^{\#\text{clusters}}$$
    \item Its conditional law given $(\sigma_v)_{v \in V}$ is obtained by picking for each edge and independent Ber$(p_e)$ random variable and setting $\omega_e = 1$ if and only if $\sigma_v = \sigma_w$ and this random variable is equal to $1$. 
    \end{enumerate}
    \end{thm}

    Now, consider the zero boundary discrete GFF (DGFF) $\phi^G$ on a graph $G = (V, E)$ and let $G_i$ be the subgraph induced by interior vertices. One can interpret its sign field $\sigma^G(v) := \operatorname{sign}(\phi^G(x))$ defined on the graph $G_i$ as a random field Ising model, with temperature $1$ and the random coupling constants given by $J_e = |\phi^G(v)\phi^G(w)|$ for any edge $e = (v,w)$. A nice observation of \cite{LupuWerner16Ising} is that the sign clusters of the naturally related metric graph GFF $\tilde{\phi}^G$ form a FK-representation of this random field Ising model:

    \begin{prop}[FK representation via metric graph GFF]
    Consider a zero boundary DGFF $\phi^G$ on a graph $G=(V,E)$ and the zero boundary metric graph GFF $\tilde{\phi}^G$ obtained by extrapolating $\phi^G$ to the line graph $\tilde G$ of $G$ using independent Brownian bridges over each edge $(v,w)$ of time-length $1$ and endpoints $\phi^G(v), \phi^G(w)$. 
    
    Consider now the bond percolation $(\omega_e)_{e\in E_i}$ on $G_i$ where we set $\omega_e = 1$ if and only if $\tilde{\phi}^G$ has the same sign throughout the edge $e$. Then $\omega_e$ forms a FK-representation of the random field Ising model $\operatorname{sign}(\phi^G)$. 
    \end{prop}

    Recall from Theorem \ref{thm:conv a} that the sign clusters of the metric graph GFF $\tilde \phi_n$, defined on the lattice approximation of a simply connected domain $D$ and converging to a continuum GFF $\Phi^D$, converge to our excursion decomposition of $\Phi^D$. This joined together with the following proposition explains why it is justified to call our excursion decomposition the FK-representation of the continuum GFF. Notice that interestingly this FK-representation is not measurable w.r.t. the field in the discrete, yet becomes measurable in the limit!

    \begin{prop}[Continuum GFF as a limit of random field Ising model]\label{prop:GFFspin}
    
Consider a sequence of lattice graphs $ D_n\subseteq (2^{-n}\Z)^2$ converging to a bounded and simply connected domain $D\subseteq \C$ in the sense that their complements inside some large box $[-C, C]^2 \supseteq D$ converge in the Hausdorff topology. 

Let $\phi_n$ be the zero boundary DGFF defined on $D_n$, set $c_1 := \sqrt{\pi/2}$ and define $s_n(v) := c_1\sqrt{\E\left[\phi_n(v)^2\right]}\sgn{\phi_n}$ to be sign of the DGFF defined in the interior of $D_n$. Then for any continuous bounded compactly supported $f$ on $D$, if we denote by $f_n$ its restriction to the interior of $D_n$, we have that as $n \to \infty$, $$\E\left[(\phi_n - s_n, f_n)^2\right] \to 0.$$
Here $(f_n,g_n) := n^{-2}\sum_{v \in D_n}f_n(v)g_n(v)$ is chosen such that it converges to the continuum inner product.

Further, if $\phi$ is the zero boundary GFF on $D$ and $\phi_n \to \phi^D$ in probability in $H^{-1-\eps}(D)$ (when properly interpolated), then also $s_n \to \phi^D$ in probability in $H^{-1-\eps}(D)$.
    \end{prop}

This follows from a simple computation, based on the following elementary lemma: 

\begin{lemma}
Let $X, Y$ be jointly Gaussian with variance $1$ and correlation $\rho \ll 1$. Then $\E(\sgn(X)\sgn(Y)) = \frac{2}{\pi}\rho + O(\rho^2)$ and $\E(X\sgn(Y)) = \sqrt{\frac{2}{\pi}}\rho + O(\rho^2)$.
\end{lemma}

\begin{proof}
The density of $(X,Y)$ at $(x,y)$ is given by $$\frac{1}{2\pi\sqrt{1-\rho^2}}\exp\left(-\frac{1}{2(1-\rho^2)}(x^2+y^2-\rho xy)\right).$$
For $\rho \ll 1$, we can write it as a perturbation of the independent vector
$$\frac{1}{2\pi\sqrt{1-\rho^2}}\exp\left(-\frac{1}{2(1-\rho^2)}(x^2+y^2)\right)(1+\rho xy + O(\rho^2x^2y^2)).$$
We can now directly calculate 
$$\E(\sgn(X)\sgn(Y)) = \rho \E(XY/|XY|) = \rho \E(|\tilde X|)^2 + O(\rho^2),$$
where $\tilde X$ is a Gaussian of variance $1-\rho^2$. This gives us 
$$\E(\sgn(X)\sgn(Y)) = \frac{2}{\pi}\rho+O(\rho^2).$$
The other calculation can be done similarly or by writing $Y = \rho X + Z$ with $Z$ and independent Gaussian of variance $1-\rho^2$.
\end{proof}

One could in fact avoid the lemma above by using an explicit formula for the correlation of signs of joint Gaussians: $\E(\sgn(X)\sgn(Y)) = \frac{2}{\pi}\arcsin(\rho)$, with notations as in the lemma. However, this lemma also generalizes to the case of the angle of a vector-valued GFF, see the remark just after the proof.
\begin{proof}[Proof of Proposition \ref{prop:GFFspin}]

We start by noting that $\E\left[(\phi_n - s_n, f_n)^2\right]$ is equal to
$$n^{-2}\sum_{v,w \in D_n}f_n(v)f_n(w)\E\left[\phi_n(v)\phi_n(w) + s_n(v)s_n(w) - \phi_n(v)s_n(w)-\phi_n(w)s_n(v) \right].$$
But now $(\phi_n(v), \phi_n(w))$ is a Gaussian vector with variance $\E\left[\phi_n(v)^2\right], \E\left[\phi_n(w)^2\right]$ and correlation $\E\left[\phi_n(v)\phi_n(w)\right]$ given by the zero boundary Green's function. In particular, as the former grow like $c\log N$ and the latter remains bounded for $\|v-w\|_2 > \eps $ for any $\eps > 0$, we can apply the lemma outside of the near-diagonal $\|v-w\|_2 \leq \eps$ to obtain that the sum over $\|v-w\|_2 \geq \eps$ is bounded by $c_\eps (\log n)^{-1}$ and thus converges to zero as $n \to \infty$. The diagonal part can be bounded by $O(\eps^2\log |\eps|)$ by a direct calculation. As this holds for any $\eps > 0$, we obtain the first claim.

To see the final part of the proposition, notice that we can similarly bound the $H^{-1}$-norm of $s_n$ interpolated sufficiently nicely over the squares (e.g. linearly over edges and the harmonically inside the squares). Indeed, denoting this interpolation by $\tilde s_n$, its expected squared $H^{-1}$-norm is given by $$\int_D\int_D G(z,w) \E [\tilde s_n(z) \tilde s_n(w)]dzdw $$ and we can use the lemma above to bound it uniformly in $n$. This gives tightness in $H^{-1-\eps}$ and thus the claim follows. 
\end{proof}

    \begin{rem}
  As mentioned above, if one considers vector-valued DGFFs, i.e. a vector $(\phi_n^1, \dots, \phi_n^d)$ of independent GFFs, one can generalize the lemma above and then the same proof shows that the angle of the DGFF also converges to the continuum vector-valued GFF, with $c_1$ replaced by another constant $c_d$. This adds yet another layer to the connection between the spin $O(N)$-models and vector-valued GFF, see e.g. \cite{AGS} for the usefulness of such connections.
    \end{rem}
    
 \subsection{Conjectured limit of the excursion decomposition of the DGFF}

 We finish the article by discussing the scaling limit of the excursion decomposition of the discrete GFF. Let us start by stating the conjecture in a slightly informal way (a precise statement would be similar to Theorem \ref{thm:conv a}). 
 
 The set-up is as follows. We consider a sequence of lattice graphs $ D_n\subseteq (2^{-n}\Z)^2$ converging to a bounded and simply connected domain $D\subseteq \C$ in the sense that their complements inside some large box $[-C, C]^2 \supseteq D$ converge in the Hausdorff topology. As before, we denote by $\phi_n$ the zero boundary DGFF defined on $D_n$. We will write $(E^n_k, \theta^n_k, \mu^n_k)_{k \geq 1}$ for the excursion decomposition of $\phi_n$, where the components the denote the sign excursion clusters, their signs and the DGFF restricted to them. 

 \begin{con}\label{conj:dgff}
 Consider a sequence of lattice graphs $ D_n\subseteq (2^{-n}\Z)^2$ converging to a bounded and simply connected domain $D\subseteq \C$ in the sense that their complements inside some large box $[-C, C]^2 \supseteq D$ converge in the Hausdorff topology. 

Let $\phi_n$ be the zero boundary DGFF defined on $D_n$, such that $\phi_n \to \Phi$ almost surely in, say, $H^{-\eps}$ and let $(E^n_k, \theta^n_k, \mu^n_k)_{k \geq 1}$ be the excursion decompositions of the DGFFs. Then these converge to a decomposition $(E_k, \theta_k, \mu_k)_{k \geq 1}$ described as follows:
\begin{itemize}
\item The union of outermost positive clusters is given by $\A_{-\lambda}$ and the union of outermost negative clusters by $\A_{\lambda}$. Each individual cluster is given by taking $\A_{-\lambda}$ or $\A_{\lambda}$ in the holes of $\A_{-\lambda, \lambda}$ of sign $\lambda$ or $-\lambda$ respectively.
\item Further clusters are defined recursively: in the holes surrounded by each negative cluster (i.e. a hole with boundary value $\lambda$) a positive cluster is given by $\A_{-\lambda}$ and in the holes surrounded by a positive cluster we obtain negative clusters by taking $\A_{\lambda}$.
\end{itemize}
Further, given the clusters, the signs are determined up to a global multiplication by $-1$ and the sign excursions $(\mu_k)_{k \geq 1}$ are given by the Minkowski content measures of $(E_k)_{k \geq 1}$ in the same gauge as in Theorem \ref{thm:prop}
 \end{con}

Notice that $\A_{-\lambda}$ for a GFF with boundary value $\lambda$ has the same law as $\A_{-2\lambda}$ of a zero boundary GFF. The reason why the boundary values are  equal to $\pm \lambda$ for the sign clusters is the same as why the height gap appears in \cite{SchSh}.

The heuristic for the conjecture goes as follows. Consider the discrete GFF with zero boundary conditions say on the triangular lattice in some domain. Then it was shown in \cite{SchSh} that the level line from $x$ to $y$ on the domain boundary converges to SLE$_4(-1,-1)$. Moreover, it was observed that one can have the joint convergence of level lines between any pair of a countable collection of boundary points. In \cite{AS2} it was noticed that this collection of level lines would be equal to $\A_{-\lambda, \lambda}$. Thus the "outer" boundaries of the cluster connected to the boundary should be given by $\A_{-\lambda, \lambda}$. Let us now consider the positive cluster. It has boundary values $\lambda$ in the limit. In the discrete, to obtain the discrete cluster we could start by taking the metric graph sign cluster connected to this boundary. After exploring this, we would then again see a zero boundary GFF but now for a metric graph. To obtain the discrete sign cluster, we would need to continue to explore, but we would be in the same situation as in the beginning - we would need to explore the boundaries of sign clusters for a zero boundary GFF. In the continuum limit this should correspond to another copy of $\A_{-\lambda,\lambda}$, although this does not directly follow from \cite{SchSh}. Iterating this way we get the following continuum description: the outermost positive cluster is given by taking $\A_{-\lambda, \lambda}$, then FPS $\A_0$ in the loops and then again $\A_{-\lambda, \lambda}$ in all the components that now have zero boundary condition. These steps are iterated until there are no boundaries with $0$ or $\lambda$ boundary value. It follows from uniqueness of first passage sets, Theorem \ref{thm:FPS}
that this set is exactly equal to $\A_{-\lambda}$. Similar considerations support the next steps of the conjecture.

\bibliographystyle{alpha}	
\bibliography{biblio}

\end{document}